\documentclass[12pt,reqno]{amsart}

\usepackage[T1]{fontenc}
\usepackage[utf8]{inputenc}
\usepackage{amsmath,amsfonts,amsthm,amssymb,amsxtra}
\usepackage{bbm} % to get \1
\usepackage[english]{babel} 
\usepackage{mathtools}
\usepackage{bm}
\usepackage{esvect}
\usepackage{enumerate}   
\usepackage[normalem]{ulem}
\usepackage{xcolor}
\usepackage{comment}

\usepackage[parfill]{parskip}    % Activate to begin paragraphs with an empty line rather than an indent

\newtheorem{theorem}{Theorem}[section]
\newtheorem{proposition}[theorem]{Proposition}
\newtheorem{lemma}[theorem]{Lemma}
\newtheorem{corollary}[theorem]{Corollary}

\theoremstyle{definition}

\theoremstyle{remark}

\newtheorem{remark}[theorem]{Remark}
\newtheorem{remarks}[theorem]{Remarks}

\numberwithin{equation}{section}

\renewcommand{\epsilon}{\varepsilon}
\newcommand{\eps}{\varepsilon}

\newcommand{\N}{\mathbb{N}}

\newcommand*\diff{\mathop{}\!\mathrm{d}} % d in dx im Integral

\newcommand{\R}{\mathbb{R}}

\newcommand{\ieps}{{i,\eps}}
\newcommand{\jeps}{{j,\eps}}
\newcommand{\keps}{{k,\eps}}

\newcommand{\oeps}{{1,\eps}}

\newcommand{\sph}{{\mathbb S^{N-1}}}
\newcommand{\ds}{\displaystyle}

\begin{document}

\title[Multibubble blow-up analysis in three dimensions]{Multibubble blow-up analysis for the Brezis-Nirenberg problem in three dimensions}

\author{Tobias K\"onig}
\address[Tobias K\"onig]{Institut für Mathematik, 
Goethe-Universität Frankfurt, 
Robert-Mayer-Str. 10, 60325 Frankfurt am Main, Germany}
\email{koenig@mathematik.uni-frankfurt.de}

\author{Paul Laurain}
\address[Paul Laurain]{Departement de Mathématiques, Université Gustave Eiffel, France}
\email{paul.laurain@univ-eiffel.fr}

\thanks{\copyright\, 2025 by the authors. This paper may be reproduced, in its entirety, for non-commercial purposes.\\
%\emph{Date}: \today \\
Partial support through ANR BLADE-JC ANR-18-CE40-002 is acknowledged. The authors are grateful to Shuibo Huang for valuable comments on a previous version of this manuscript, which led to its improvement.}

\begin{abstract}
For a smooth bounded domain $\Omega \subset \R^3$ and smooth functions $a$ and $V$, we consider the asymptotic behavior of a sequence of positive solutions $u_\epsilon$ to $-\Delta u_\epsilon + (a+\epsilon V) u_\epsilon = u_\epsilon^5$ on $\Omega$ with zero Dirichlet boundary conditions, which blow up as $\epsilon \to 0$. We derive the sharp blow-up rate and characterize the location of concentration points in the general case of multiple blow-up,  thereby obtaining a complete picture of blow-up phenomena in the framework of the Brezis-Peletier conjecture in dimension $N=3$.

\end{abstract}

\maketitle

\section{Introduction}
\label{section introduction}

For an open bounded set $\Omega \subset \R^3$, let us consider a sequence of solutions $(u_\eps)$ to the problem
\begin{align}
-\Delta u_\eps + (a + \eps V) u_\eps &=  u_\eps^5
 \qquad \text{ on } \Omega, \nonumber \\
 u_\eps &> 0 \qquad \text{ on } \Omega, \label{brezis peletier additive}  \\
 u_\eps &= 0 \qquad \text{ on } \partial \Omega. \nonumber
\end{align}
We will assume throughout the paper that $a \in C(\overline{\Omega}) \cap C^{1, \sigma}_\text{loc}(\Omega)$ for some $\sigma \in (0,1)$ and $V \in C(\overline{\Omega}) \cap C^1(\Omega)$, but it is perfectly meaningful to think of $a$ and $V$ being constants. Moreover, we always assume that $-\Delta + a$ is coercive, which is in fact a necessary condition in this context, see Appendix C of \cite{KL2}, and that the boundary of $\Omega$ is $C^2$.

The study of this equation has been initiated in the seminal work \cite{Brezis1983} by Brezis and Nirenberg. The understanding of the behavior of solutions of this equation is pivotal in the Yamabe problem, see for instance \cite{DH05} and reference therein. Subsequently, Brezis and Peletier \cite{Brezis1989} initiated the study of \eqref{brezis peletier additive} in the case where there is at least one \emph{blow-up point} $x_0 \in \overline{\Omega}$, i.e. there is a sequence $x_\eps \to x_0$ such that $u_\eps (x_\eps) \to \infty$ as $\eps \to 0$. In \cite{Brezis1989} the authors conjecture an asymptotic expression for $\|u_\eps\|_\infty$ in the case where $(u_\eps)$ has precisely one blow-up point.

To discuss this in more depth, let us introduce the object that largely governs the asymptotic behavior of $(u_\eps)$, namely the Green's function $G_a: \Omega \times \Omega \to \R$. This is the unique function satisfying, for each fixed $y\in\Omega$,
\begin{equation} \label{Ga-pde}
\left\{
\begin{array}{l@{\quad}l}
-\Delta_x\, G_a(x,y) + a(x)\, G_a(x,y) =  \delta_y & \quad \text{in} \ \ \Omega\,, \\
G_a(\cdot,y) = 0  & \quad \text{on} \ \ \partial\Omega \,.
\end{array}
\right.  
\end{equation}
Note that $G_a(x,y) > 0$ for every $x,y \in \Omega$ as a consequence of coercivity. 
The regular part $H_a$ of $G_a$ is defined by 
\begin{equation} \label{ha-def}
H_a(x,y) := \frac{1}{4\pi |x-y|} - G_a(x,y)\, .
\end{equation}
It is well-known that for each $y\in\Omega$ the function $H_a(\cdot,y)$, which is originally defined in $\Omega\setminus\{y\}$, extends to a continuous function in $\Omega$. Thus we may define the \emph{Robin function}
$$
\phi_a(y) := H_a(y,y) \,.
$$
It is proved in \cite{Brezis1989, Rey1989, Han1991} that single-blow-up sequences of solutions to \eqref{brezis peletier additive} must concentrate at critical points $x_0$ of $\phi_a$. 

For space dimension $N \geq 4$ and $a \equiv 0$, $V \equiv -1$, sequences of solutions with a single blow-up point $x_0$ exist as a consequence of the Brezis--Nirenberg existence result \cite{Brezis1983}. As observed conjecturally in \cite{Brezis1989} and confirmed rigorously in \cite{Rey1989, Han1991}, the blow-up behavior of such $u_\eps$ is governed by the value $\phi_0(x_0)$ in the sense that 
	\begin{equation}
	\label{lim-eps-u^2}
	 \lim_{\epsilon \to 0}\,  \epsilon \,  \|u_\eps\|^\frac{2(N-4)}{N-2}_\infty   = d_N  \phi_0(x_0), 
\end{equation}
where $d_N$ is a constant only depending on $N$ only. (Note that $\phi_0(x_0) > 0$ by the maximum principle for $H_0(x_0, \cdot)$.) 

However, the conjectures in \cite{Brezis1989} leave open what happens in the Brezis--Nirenberg-critical dimension $N=3$, even in the case of one blow-up point $x_0$. Indeed, by \cite{Druet2002} single blow-up sequences must satisfy $\phi_a(x_0) = 0$ in that case, so that the leading order of $\|u_\eps\|_\infty$ can no longer be captured by the right side of \eqref{lim-eps-u^2}. In particular, since $\phi_0 > 0$ we necessarily must have $a \not \equiv 0$ if blow-up occurs. 

Dealing with the case where $\phi_a(x_0) = 0$ in the context of blow-up asymptotics is a formidable problem. The reason for this is that the value $\phi_a(x_0)$ appears as the leading coefficient of a certain energy expansion related to $u_\eps$, see e.g. \cite[eq. (23)]{Rey1989}, \cite[eq. (3.28)]{Frank2021b}. As long as $\phi_a(x_0) > 0$, this term determines the asymptotic behavior of $\|u_\eps\|_\infty$ as in \eqref{lim-eps-u^2}. Now if $\phi_a(x_0) = 0$, it is the next term in this expansion that becomes relevant for the asymptotics of $\|u_\eps\|_\infty$. To extract this term, the expansion needs in turn to be computed at a higher precision, which is a considerable analytic challenge.

Indeed, even for a single blow-up point $x_0$, the asymptotics replacing \eqref{lim-eps-u^2} in case $N = 3$ (and hence $\phi_a(x_0)=0$) have been derived only recently in \cite{Frank2021b} under a non-degeneracy assumption on $\phi_a$, see also \cite{Frank2021, Frank2019b} for the special case of least-energy solutions. They read
\begin{equation}
\label{lim eps u infty^2}
\lim_{\epsilon \to 0}\,  \epsilon \,  \|u_\eps\|^2_\infty   = \frac{\sqrt 3}{4} \frac{|a(x_0)|}{|\int_\Omega V(y) G_a(x_0,y)^2 \diff y|}. 
\end{equation} 
(In this statement it is assumed that $a(x_0) < 0$. If at the same time $\int_\Omega V(y) G_a(x,y)^2 \diff y = 0$, then \eqref{lim eps u infty^2} continues to hold with right side equal to $+ \infty$.)

\section{Main results}

In this paper we achieve a complete analysis of blowing-up solutions $u_\eps$ to \eqref{brezis peletier additive} in the spirit of Brezis and Peletier, in the general case where the sequence of solutions $u_\eps$ to \eqref{brezis peletier additive} may present multiple (a priori even infinitely many) blow-up points.

In particular, we are able to describe precisely the $L^\infty$ asymptotics near each concentration point, generalizing \eqref{lim eps u infty^2}. The appropriate expression, see  \eqref{mu asymptotics thm} below, involves an interaction between the blow-up points through new quantities and cannot be guessed easily from \eqref{lim eps u infty^2}.

To state our result precisely, we introduce some more notation. 

For any number $n \in \N$ of concentration points, let 
\[ \Omega_\ast^n := \{ \bm{x}= (x_1,...,x_n)\in \Omega^n \, : \, x_i \neq x_j \text{  for all  } i \neq j \}. \]
 For $\bm{x} \in \Omega_\ast^n$ we denote $M_a(\bm{x}) \in \R^{n \times n} = (m_{ij})_{i,j=1}^n$ the matrix with entries 
\begin{equation}
\label{m ij definition}
m_{ij}(\bm{x}) := 
\begin{cases}
\phi_a(x_i) & \text{ for } i = j, \\
- G_a(x_i, x_j) & \text{ for } i \neq j. 
\end{cases}
\end{equation}
Its lowest eigenvalue $\rho_a(\bm{x})$ is simple and the corresponding eigenvector can be chosen to have strictly positive components, see Lemma \ref{lemma perron frobenius}. We denote by $\bm{\Lambda}(\bm{x}) \in \R^n$ the unique vector such that 
\[ M_a(\bm{x}) \cdot \bm{\Lambda}(\bm{x}) = \rho_a(\bm{x})\bm{\Lambda}(\bm{x}) , \qquad (\bm{\Lambda}(\bm{x}) )_1 = 1. \]

Moreover, we define the Aubin--Talenti type bubble function 
\[ B(x):= \left(1 + \frac{|x|^2}{3} \right)^{-1/2} \]
and, for every $\mu > 0$ and $x_0 \in \R^3$ its rescaled and translated versions 
\[ B_{\mu, x_0}(x) = \mu^{-1/2} B\left( \frac{x - x_0}{\mu} \right) = \frac{\mu^{1/2}}{(\mu^2 + \frac{|x-x_0|^2}{3})^{1/2}}. \]
Notice that the normalizations are chosen here so that $-\Delta B_{\mu, x_0} = B_{\mu,x_0}^5$ on $\R^3$, for every $\mu > 0$ and $x_0 \in \R^3$.

The $B_{\mu, x_0}(x)$ are easily found to represent the leading order profile of $u_\eps$ around each of its concentration points, see Proposition \ref{proposition preliminaries multi} below. However, as explained above, a higher precision is needed for our purposes. Thus we shall need to introduce the following explicit correction function. 
For $j,k \in \{1,2,3\}$, we consider functions $W_{jk}$ which satisfy
\begin{equation}
\label{w equation}
\begin{cases}
-\Delta W_{jk} - 5 W_{jk} B^4 = 0 , \quad W_{jk}(x) = x_j x_k + o(|x|^2) \quad \text{ as } x \to 0 & \text{ if } j \neq k, \\
-\Delta W_{jj} - 5 W_{jj} B^4 = -B , \quad W_{jj}(x) = \frac{1}{2} x_j^2 + o(|x|^2) \quad \text{ as } x \to 0 & \text{ if } j = k.
\end{cases}
\end{equation} 
We construct these functions in Lemma \ref{lemma Wjk} below. 
%\begin{equation}
%\label{W definition}
%W_{\mu, x_0}(x) := \mu^{-1/2} W\left(\frac{x - x_0}{\mu}\right). 
%\end{equation} 

Here is our main result. 

\begin{theorem}
\label{theorem multibubble}
Let $(u_\eps)$ be a sequence of solutions to \eqref{brezis peletier additive} with $\|u_\eps\|_\infty \to \infty$. Then there exists $n \in \N$ and $n$ sequences of points $x_\oeps,...,x_{n,\eps} \in \Omega$ such that $\mu_\ieps:= u_\eps(x_\ieps)^{-2} \to 0$ as $\eps \to 0$ and $\nabla u_\eps (x_\ieps) = 0$ for every $\eps > 0$.  

Moreover, the following holds. 
\begin{enumerate}[(i)]
\item \label{item conc points}  \textbf{Properties of  concentration points: } There is $\bm{x}_0 := (x_{1,0},...,x_{n,0}) \in \Omega^n_\ast$ such that up to a subsequence, $(x_\oeps,...,x_{n, \eps}) \to  \bm{x}_0$. Moreover, $\rho_a(\bm{x}_0) = \nabla_{\bm{x}} \rho_a(\bm{x}_0) = 0$. The matrix $M_a(\bm{x}_0)$ is semi-positive definite with simple lowest eigenvalue $\rho_a(\bm{x}_0)$. The associated eigenvector is $\bm{\Lambda}(\bm{x}_0) = (\Lambda_{1,0},...,\Lambda_{n,0})$ with $\Lambda_{i,0} = \lim_{\eps \to 0} \frac{\mu_\ieps^{1/2}}{\mu_\oeps^{1/2}} \in (0,\infty)$ for every $i$. 
\end{enumerate}
 
\begin{enumerate}[(i)]
 \setcounter{enumi}{1}
\item \label{item global} \textbf{Global asymptotics:} $\mu_{1,\eps}^{-1/2} u_{\eps}(x) \to 4 \pi \sqrt 3 \sum_i \Lambda_{i,0} G_a(x_i, x) =: \mathcal G(x)$ uniformly away from $\{x_{1,0},...,x_{n,0}\}$.

\item \label{item local} \textbf{Refined local asymptotics:} Let $B_\ieps := B_{\mu_\ieps, x_\ieps}$ and 
\[ W_\ieps(x):=  \sum_{j,k} \left(\partial_{jk} u_\eps(x_\ieps) - \partial_{jk} B_\ieps(x_\ieps) \right) W_{jk}\left(\frac{x - x_\ieps}{\mu_\ieps} \right). 
\]

 Then, for $\delta > 0$ small enough and every $0 < \nu < 1$, 
\begin{equation}
\label{u asymptotics thm}
|u_\eps - B_\ieps - \mu_\ieps^2 W_\ieps| \lesssim  \mu_\eps^{\frac 12 - \nu} |x-x_\ieps|^{2 + \nu} \quad \text{ on } B(x_\ieps, \delta). 
\end{equation} 
\item \label{item blowup} \textbf{Blow-up rate:} Assume either \textbf{(a)} that $\rho_a$ is $C^2$ in $x_0$ with $D^2_{\bm{x}} \rho_a(\bm{x}_0) \geq c$ for some $c >0$, in the sense of quadratic forms, or \textbf{(b)} that $\rho_a$ is real-analytic in $\bm{x}_0$. Let $\mathcal G$ be as in \eqref{item global}. Then 
\begin{equation}
\label{mu asymptotics thm}
\lim_{\eps \to 0} \eps u_\eps(x_\ieps)^2  = 12 \pi^2 \sqrt 3 \Lambda_{i,0}^{-2} \frac{\sum_{j=1}^n a(x_{j,0}) \Lambda_{j,0}^4}{\int_\Omega V \mathcal G^2 \diff x},
\end{equation}
provided that both of the quantities $\sum_j a(x_{j,0}) \Lambda_{j,0}^4$ and $\int_\Omega V \mathcal G^2 \diff x$ are non-zero. If one of them equals zero, but not the other one, \eqref{mu asymptotics thm} remains true, with the right side  being equal to $0$, respectively $+\infty$.  
\end{enumerate}   
\end{theorem}

We emphasize that contrary to previous result on the Brezis-Peletier conjecture no bound on the number of blow-up points of $u_\eps$ is assumed. In fact it was already known from the work of Li-Zhu \cite{Li1999} that in dimension $3$ the blow-up points must be isolated, see also \cite{Li1995,Li1996}. In fact, points (i) and (ii) of the theorem follow  directly from \cite{Druet2010} and have been reproved in \cite{KL2}, see also reference therein, but we chose to include them in the theorem for a complete statement.

This theorem is an exhaustive description of blow-up phenomena of equation \eqref{brezis peletier additive} in dimension $N = 3$. Its main points are items (iii) and (iv), namely the strong (superquadratic) pointwise bound \eqref{u asymptotics thm} of $u_\eps$ near each blow-up point, and the  explicit asymptotic expression for the blow-up rates $u_\eps(x_\ieps)$ in \eqref{mu asymptotics thm} derived from it.

Let us give several more remarks to put this result into context. 

\begin{remarks}
\begin{enumerate}[(a)]

\item The appearance of the matrix $M_a$ and its lowest eigenvalue $\rho_a$ in the asymptotic expansions relative to multiple blow-up is well-known, see e.g. \cite{Bahri1995, Musso2002, Musso2018, Cortazar2020, Malchiodi2021}. E.g. in \cite{Musso2018} solutions to \eqref{brezis peletier additive} blowing up in points $(x_1,...,x_n) = \bm x$  are constructed under the assumption $\rho_a(\bm{x}_0) = \nabla \rho_a(\bm{x}_0) = 0$, which is optimal as Theorem \ref{theorem multibubble} shows.

Our new contribution is, in this context, to deal with the vanishing $\rho_a(\bm x_0) = 0$ and to extract the next-order term determining the asymptotics \eqref{mu asymptotics thm}, compare the discussion leading to \eqref{lim eps u infty^2}. This is the main difficulty overcome in our paper.
To underline the novelty of \eqref{mu asymptotics thm}, one may remark that even for the solutions constructed in  \cite{Musso2018}, only the information $\mu_\ieps = \mathcal O(\eps)$ is obtained through the existence argument, which is less precise than \eqref{mu asymptotics thm}.

\item \label{remark nondeg conditions}
The conditions (a) and (b) in item \eqref{item blowup} can be thought of as non-degeneracy assumptions. This is clear for (a), which is the natural generalization of \cite[Assumption 1.1(d)]{Frank2021b} to the case of multiple blow-up points. The $C^2$-differentiability of $\rho_a$ is guaranteed under the slightly stronger assumption $a \in C^{0,1}(\overline{\Omega}) \cap C_\text{loc}^{2,\sigma}(\Omega)$ for some $\sigma \in (0,1)$, see Lemma \ref{lemma rho_a analytic}.  

Remarkably, in assumption (b), no positivity condition at all is needed, only higher regularity of $\rho_a$, more precisely real-analyticity. This observation is new even for the case of one blow-up point. We prove in Lemma \ref{lemma phi a is analytic} that $\rho_a$ is real-analytic if $a \equiv \text{const.}$. But in view of similar results, e.g. \cite{John1950, Khenissy2010, Franceschini2021}, it is reasonable to expect that $\phi_a$, and hence $\rho_a$, is real-analytic whenever $a$ is; see also the remarks after the proof of Lemma \ref{lemma phi a is analytic}. This is an open question to the best of our knowledge, and it would be very interesting to obtain an answer to it.

\item Since Theorem \ref{theorem multibubble} only makes a statement about space dimension $N=3$, a natural question is to determine the asymptotic behavior of a sequence of solutions to $-\Delta u_\eps + \eps V u_\eps = u_\eps^\frac{N+2}{N-2}$ when $N \geq 4$. This has been completed by the authors in the recent work \cite{KL2}. (We take $a \equiv 0$ here so that $\phi_a$ remains well-defined when $N \geq 4$. But even treating more general $a \not \equiv 0$ appears equally possible by using the appropriate asymptotic expansion of $G_a(x, \cdot)$, which contains additional singular terms.)
%\ {We take $a \equiv 0$ here so that $\phi_a$ remains well-defined when $N \geq 4$. But treating more general $a \not \equiv 0$ appears equally possible by using the appropriate asymptotic expansion of $G_a(x, \cdot)$, which contains additional singular terms.} We are not aware of a result like this in the literature, even though in view of the above discussion it should be easier to obtain than Theorem \ref{theorem multibubble}. Analogous computations as those used in the proof of Theorem \ref{theorem multibubble} lead us to believe that in this case (say $N \geq 5$ for simplicity)
%\begin{equation}
%\label{equation higher dim}
%( M_0(\bm{x}_0) \cdot \bm{\lambda}_0)_i = -c_N ( \lim_{\eps \to 0} \eps \mu_\ieps^{-N+4}) V(x_{i,0}) \lambda_{i,0}
%\end{equation} 
%for some dimensional constant $c_N > 0$. Here, 
%\[ \mu_\ieps^{-\frac{N-2}{2}} = u(x_\ieps) \quad  \text{ and } \quad \lambda_{i,0} = \lim_{\eps \to 0} \left(\frac{\mu_\ieps}{\mu_\oeps}\right)^\frac{N-2}{2} \in (0, \infty),
%\] and we have employed the above notations otherwise. 
%We will come back to a thorough analysis of this problem in future work. 

\item Another problem closely related to \eqref{brezis peletier additive} is (for $N \geq 3$)
\begin{equation}
\label{bp subcritical}
-\Delta u_\eps + a u_\eps = u_\eps^{\frac{N+2}{N-2}-\eps}, \qquad u_\eps > 0, \qquad u_\eps|_{\partial \Omega} = 0,
\end{equation}
whose single-blow-up asymptotics as $\eps \to 0+$ in the case $a \equiv 0$ have as well been determined by \cite{Brezis1989, Rey1989, Han1991}, see also \cite{Hebey2000}. 

The case of multiple concentration points in \eqref{bp subcritical}, still for $a \equiv 0$, has subsequently been studied in \cite{Rey1991, Bahri1995, Rey1999}. There, the authors derive an asymptotic formula for (essentially) $u_\eps(x_\ieps)$ similar to \eqref{item blowup} under the condition that  $\rho_0(\bm{x}_0) > 0$. In the spirit of our above discussion, this should be viewed as the analogue of the simpler case \eqref{lim-eps-u^2}. For single-blow-up in $N =3$, the analogue of the harder formula \eqref{lim eps u infty^2} has been proved in \cite[Theorem 1.3]{Frank2021b}. (A subtle, yet interesting difference between \eqref{brezis peletier additive} and \eqref{bp subcritical} is that when single blow-up happens (say in $x_0$), one automatically has $\phi_a(x_0) = 0$  in the former problem \cite[Theorem 1.5]{Frank2021b}, but not in the latter \cite[Theorem 2.(b)]{delPino2004}. Hence $\rho_0(\bm{x}_0) > 0$ may well be satisfied, even when $N = 3$.)

On the other hand, for multiple blow-up in \eqref{bp subcritical}, we are not aware of a formula analogous to \eqref{mu asymptotics thm} for $\rho_a(\bm{x}_0) = 0$, not even when $a \equiv 0$. We believe that our methods can yield such a formula, but we leave this question to future work. Such belief may be justified by the results in \cite{Frank2021b}, where $L^\infty$ single-blow-up asymptotics are obtained for both \eqref{brezis peletier additive} and \eqref{bp subcritical} when $\phi_a(x_0) = 0$ by arguments very similar to each other, which are however different from the ones employed here.
\end{enumerate}
\end{remarks}

A crucial tool which we use repeatedly in the asymptotic analysis leading to Theorem \ref{theorem multibubble} is the non-degeneracy of the bubble $B$ as a solution to the equation $-\Delta u = u^5$. The non-degeneracy property roughly says that the solutions of the linearized equation around $B$, i.e. $-\Delta v = 5 B^4 v$, with polynomial growth are either the one you may expect, that is to say the one you can construct from the family $B_{\mu,x_0}$, or some function equivalent to the same homogeneous polynomial both at $0$ and $+\infty$. In fact this non-degeneracy property does not depend on the fact that $N = 3$, we state it, see Proposition \ref{u linear combi theorem} in the appendix, for general dimension $N \geq 3$.  Proposition \ref{u linear combi theorem} substantially improves previous statements of the same kind with decreasing behavior at infinity, see for example in \cite[Appendix D]{Rey1990} and \cite[Lemma 2.4]{Chen1998}, the result is derived assuming that $\nabla v \in L^2(\R^N)$, respectively that $|v(x)| = o(1)$ as $|x| \to \infty$. However, a kind of analysis suitable for our purposes was already performed by Korevaar, Mazzeo, Pacard and Schoen, see Section 2.2 of \cite{KMPS} and references therein. In their setting they deal with singular decreasing solutions of the linearization of $-\Delta u = u^\frac{N+2}{N-2}$ about a singular solution, nevertheless the proof contains all ingredients to be applied in our setting. The full needed non-degeneracy statement and a sketch of its proof is postponed to Appendix \ref{A}.

\subsection{Structure of the paper. }

%In Section \ref{section non-degeneracy} below we are going to prove Theorem \ref{theorem non degeneracy intro} and Corollary \ref{corollary u = 0} through a decomposition with respect to spherical harmonics and some ODE analysis. 

Section \ref{section asymp analysis additive} is devoted to the proof of part \eqref{item local} of Theorem \ref{theorem multibubble}. Our starting point is a qualitative result on blow-up sequences from \cite{Druet2010}, Proposition \ref{proposition preliminaries multi}, which we refine in two iteration steps. 

In Section \ref{section proof of thm}, the precise expansion of $u_\eps$ from Section \ref{section asymp analysis additive} is used in turn to derive two asymptotic Pohozaev-type identities involving $\rho_a(\bm{x}_\eps)$ and $\nabla \rho_a(\bm{x}_\eps)$, respectively. Together with some linear-algebraical arguments on the matrix $M_a(\bm{x}_\eps)$ and using either the non-degeneracy assumption $D^2 \rho_a(\bm{x}_0) \geq c$ or the analyticity of $\rho_a$, the combination of these identities yields the asymptotic expression of $u_\eps(x_\ieps)$ claimed in part \eqref{item blowup} of Theorem \ref{theorem multibubble}. 

In Appendix \ref{A}, as already mentioned above, we give some details concerning the non-degeneracy property of the limit equation $-\Delta u = u^\frac{N+2}{N-2}$ under polynomial growth conditions (for general $N = 3$).

Finally, a second appendix contains some explicit computations involving the functions $G_a$ and $W$. 

\subsection{Notation}

Let $f,g : X \to \R_+$ be nonnegative functions defined on some set $X$. We write $f(x) \lesssim g(x)$ if there is a constant $C > 0$ independent of $x$ such that $f(x) \leq C g(x)$ for all $m \in M$, and accordingly for $\gtrsim$. If $f \lesssim g$ and $g \lesssim f$, we write $f \sim g$. 

Let $f: X^n \to \R$ be a function of $n$ variables for $X \subset \R^3$. We write $\nabla_{x_i}$ and $\partial^{x_i}_k f$ to denote the gradient, respectively the $k$-th partial derivative, of $f$ with  the $i$-th variable. When $n = 2$, we also write $\nabla_{x_1} = \nabla_x$, $\partial^{x_i}_k = \partial_k^x$ and $\nabla_{x_2} = \nabla_y$, $\partial^{x_2}_k = \partial_k^y$.

\section{Asymptotic analysis of $u_\eps$}
\label{section asymp analysis additive}

The following proposition follow almost directly from \cite{Druet2010} and it has been reproved in this exact frame in \cite{KL2}, see proposition B.1. It is the starting point of our analysis. 

\begin{proposition}
\label{proposition preliminaries multi}
Let $(u_\eps)$ be a sequence of solutions to \eqref{brezis peletier additive}. Then, up to extracting a subsequence, there exists $n \in \N$ and points $x_{1,\eps},..., x_{n,\eps}$ such that the following holds. 
\begin{enumerate}[(i)]
\item $x_\ieps \to x_i \in \Omega$ for some $x_i \in \Omega$ with $x_i \neq x_j$ for $i \neq j$. 
\item $\mu_\ieps := u_\eps(x_\ieps)^{-2} \to 0$ as $\eps \to 0$ and $\nabla u_\eps(x_\ieps) = 0$ for every $i$. 
\item \label{item lambda i} $\lambda_{i,0} := \lim_{\eps \to 0} \lambda_\ieps := \lim_{\eps \to 0} \frac{\mu_\ieps^{1/2}}{\mu_\oeps^{1/2}}$ exists and lies in $(0, \infty)$ for every $i$. 
\item $\mu_{i,\eps}^{1/2} u_\eps(x_{i,\eps} + \mu_{i,\eps} x) \to B$ in $C^1_\text{loc}(\R^n)$. 

\item There is $C > 0$ such that $u_\eps \leq C \sum_i B_\ieps$ on $\Omega$.  
\end{enumerate}
\end{proposition}

Here and in the following, all sums are over $1,...,n$ unless specified otherwise.

Since the $\mu_\ieps$ are all of comparable size by Proposition \ref{proposition preliminaries multi}, it will be convenient in the following to state error estimates in terms of $\mu_\eps := \max_i \mu_\ieps \lesssim \min_i \mu_\ieps$.

Proposition \ref{proposition preliminaries multi} says that near $x_\ieps$, the function $u_\eps$ is well approximated by $B_\ieps$. Our goal in this section is to extract the next term in the asymptotic development of $u_\eps$ near $x_\ieps$. This term will turn out to involve the function 
\begin{equation}
\label{W ieps definition}
W_\ieps (x) = \sum_{j,k} \left(\partial_{jk} u_\eps(x_\ieps) - \partial_{jk} B_\ieps(x_\ieps) \right) W_{jk}\left(\frac{x - x_\ieps}{\mu_\ieps} \right), 
\end{equation}
with $W_{jk}$ as in \eqref{w equation}.

We also define the small ball 
\[ \mathsf{b}_\ieps := B(x_\ieps, \delta_0) \]
around $x_\ieps$, with some number $\delta_0 >0$ independent of $\eps$ and chosen so small that $\delta_0 < \frac 12  \min_{i\neq j} |x_\ieps - x_\jeps|$ for all $\eps > 0$ small enough.

Here is the main result of this section. 

\begin{theorem}
\label{theorem asmyptotic expansion additive}
Let $u_\eps$ be a sequence of solutions to \eqref{brezis peletier additive} and adopt the notations from Proposition \ref{proposition preliminaries multi}. For every $i = 1,...,n$, denote
\[ r_\ieps := u_\ieps - B_\ieps, \qquad  q_\ieps := r_\ieps - \mu_\ieps^2 W_\ieps =  u_\eps - B_\ieps - \mu_\ieps^2 W_\ieps. \]
Then, for every $x \in \mathsf{b}_\ieps$, we have the bounds
\begin{align}
\label{r eps bound theorem}
|r_\ieps(x)| &\lesssim  \mu_\eps^{\frac 12- \vartheta} |x-x_\ieps|^{1+\vartheta}, \qquad \text{ for every } 0 < \vartheta < 1, \\
\label{q eps bound theorem}
|q_\ieps(x)| &\lesssim  \mu_\eps^{\frac 12 - \nu} |x-x_\ieps|^{2 + \nu}, \qquad \text{ for every } 0 < \nu < 1.
\end{align} 
\end{theorem}

Our proof of Theorem \ref{theorem asmyptotic expansion additive} is in the spirit of \cite{Druet2004} and related works. It consists of two iterative steps carried out in Subsections \ref{subsection first bound} and \ref{subsection refined expansion} below. The structure of each step is similar: through a well-chosen asymptotic analysis ansatz, the desired bound is ultimately deduced from the non-degeneracy of solutions to some limit equation. This is precisely where  Corollary \ref{corollary u = 0} enters. We emphasize again that to obtain the precision required in \eqref{q eps bound theorem}, Corollary \ref{corollary u = 0} needs to be applied with $\tau \in (2,3)$, in which case solutions to the linearized equation \eqref{lin eq} may in general take a non-standard form like \eqref{u linear combi theorem}.

\subsection{A first quantitative bound}
\label{subsection first bound}

In a first step, we now prove the cruder one of the two bounds stated in Theorem \ref{theorem asmyptotic expansion additive}. Let us for convenience restate the  result of this subsection as follows. 

\begin{proposition}
\label{proposition r eps}
Let $i=1,...,n$. As $\eps \to 0$, for every $0 < \vartheta < 1$, 
\[ |(u_\eps - B_\ieps)(x)| \lesssim \mu_\eps^{\frac{1}{2} - \vartheta} |x-x_\ieps|^{1 + \vartheta}, \qquad \text{ for all } x \in \mathsf{b}_\ieps. \]

\end{proposition}

\begin{proof}
We denote $r_\ieps= u_\eps - B_\ieps$. We fix some $0 < \vartheta < 1$ and denote  
\begin{equation}
\label{R eps definition}
R_\ieps(x):= \frac{r_\ieps(x)}{|x-x_\ieps|^{1 + \vartheta}}.
\end{equation}

Fix some $z_\ieps \in \mathsf{b}_\ieps$ such that
\begin{equation}
\label{R eps (z eps)}
R_\ieps (z_\ieps) \geq \frac 12 \|R_\ieps \|_{L^\infty(\mathsf{b}_\ieps)}. 
\end{equation}
(Notice that $ \|R_\ieps \|_{L^\infty(\mathsf{b}_\ieps)} < \infty$, thanks to Taylor's expansion and the fact that $r_\ieps(x_\ieps) = \nabla r_\ieps(x_\ieps) = 0$ by Proposition \ref{proposition preliminaries multi}.)

Moreover, we denote $d_\ieps:= |x_\ieps - z_\ieps|$. Let us define the rescaled and normalized version
\begin{equation}
\label{bar r eps definition}
\bar{r}_\ieps(x):= \frac{r_\ieps(x_\ieps + d_\ieps x)}{r_\ieps(z_\ieps)}, \qquad x \in B(0, d_\ieps^{-1} \delta_0). 
\end{equation}
Then \eqref{R eps (z eps)} implies 
\begin{equation}
\label{r lesssim x vartheta}
\bar{r}_\ieps(x) \lesssim |x|^{1+\vartheta}, \qquad x \in B(0, d_\ieps^{-1} \delta_0),
\end{equation}
in particular $\bar{r}_\eps$ is uniformly bounded on compacts of $\R^3$. 

Abbreviating $a_\eps := a + \eps V$, we have, on $B(0, d_\ieps^{-1} \delta_0)$,
\begin{equation}
\label{bar r eps equation}
-\Delta \bar{r}_\ieps + d_\ieps^2 \bar{a}_\eps \frac{\bar{u}_\ieps}{r_\ieps(z_\ieps)} = \bar{r}_\ieps d_\ieps^2 \left(\bar{u}_\ieps^4 + \bar{u}_\ieps^3 \bar{B}_\ieps + \bar{u}_\ieps^2 \bar{B}_\ieps^2 + \bar{u}_\ieps \bar{B}_\ieps^3 + \bar{B}_\ieps^4 \right) 
\end{equation}
Here we wrote $\bar{u}_\ieps(x) :=  u_\eps(x_\ieps + d_\ieps x)$ and likewise $\bar{a}_\ieps (x) :=  a_\eps(x_\ieps + d_\ieps x)$ and $\bar{B}_\ieps(x) :=  B_\ieps(x_\ieps + d_\ieps x) = \mu_\ieps^{-1/2} B (\mu_\ieps^{-1} d_\ieps x)$. 

We treat three cases separately, depending on the ratio between $\mu_\eps$ and $d_\ieps$. 

\textit{Case 1. $\mu_\eps >> d_\ieps$ as $\eps \to 0$. } In that case, we have $\bar{B}_\ieps \lesssim \mu_\eps^{-1/2}$ uniformly on $\R^3$. Since $\bar{u}_\ieps \lesssim \bar{B}_\ieps$ on $\mathsf{b}_\ieps$, the right side of \eqref{bar r eps equation} therefore tends to zero uniformly on compacts in that case because $d_\ieps^2 \mu_\eps^{-2} \to 0$. 

Using $\bar{u}_\ieps \lesssim \bar{B}_\ieps \lesssim \mu_\ieps^{-1/2}$ and $\frac{1}{r_\ieps(z_\ieps)} \lesssim d_\ieps^{-1-\vartheta} \frac{1}{\|R_\ieps\|_\infty}$ by \eqref{R eps (z eps)}, the second summand on the left side of \eqref{bar r eps equation} is bounded by 
\begin{align*}
\left| d_\ieps^2 \bar{a}_\ieps \frac{\bar{u}_\ieps}{r_\ieps(z_\ieps)} \right| \lesssim \frac{d_\ieps^{1-\vartheta} \mu_\ieps^{-1/2}}{\|R_\ieps\|_{L^\infty(\mathsf{b}_\ieps)}}  \lesssim \frac{\mu_\eps^{1/2 - \vartheta}}{\|R_\ieps\|_{L^\infty(\mathsf{b}_\ieps)}}.
\end{align*}
Now suppose that for contradiction that $\|R_\ieps\|_{L^\infty(\mathsf{b}_\ieps)} >> \mu_\eps^{1/2 - \vartheta}$ as $\eps \to 0$. Then this term goes to zero uniformly. Thus the limit $\bar{r}_{i,0} := \lim_{\eps \to 0} \bar{r}_\ieps$ satisfies 
\[ -\Delta \bar{r}_{i,0} = 0 \qquad \text{ on } \R^3. \]
By Liouville's theorem, the growth bound \eqref{r lesssim x vartheta} implies that $\bar{r}_{i,0}(x) = b \cdot x + c $ for some  $b\in \R^3$, $c \in \R$. On the other hand, still by \eqref{r lesssim x vartheta}, we find $b = \nabla \bar{r}_{i,0}(0) = 0$ and  $c = \bar{r}_{i,0}(0) = 0$,  and  hence $\bar{r}_{i,0} \equiv 0$. But by the choice of $d_\ieps$, there is $\xi_\ieps := \frac{z_\ieps - x_\ieps}{d_\ieps} \in \mathbb S^2$ such that $\bar{r}_\ieps(\xi_\ieps) = 0$. Up to a subsequence, $\xi_{i,0} := \lim_{\eps \to 0} \xi_\ieps \in \mathbb S^2$ exists and satisfies $\bar{r}_{i,0}(\xi_{i,0}) = 1$. This contradicts $\bar{r}_{i,0} \equiv 0$. 

Thus we must have $\|R_\ieps\|_{L^\infty(\mathsf{b}_\ieps)} \lesssim \mu_\ieps^{\frac 12 - \vartheta}$, i.e. $r_\ieps(x) \lesssim \mu_\ieps^{\frac{1}{2} - \vartheta} |x-x_\ieps|^{1 + \vartheta}$. 

\textit{Case 2.a) $\mu_\eps << d_\ieps <<1$ as $\eps \to 0$. } This case works similarly, but we need to argue a little more carefully. This is because the relevant bound $\bar{B}_\ieps \lesssim \mu_\eps^{1/2} d_\ieps^{-1}|x|^{-1} \lesssim \mu_\eps^{1/2} d_\ieps^{-1}$ now only holds on compacts of $\R^3 \setminus \{0\}$ and no convergence holds at the origin. Nevertheless, we have 
\[ \bar{r}_\ieps d_\ieps^2 (\bar{u}_\ieps^4 + ... + \bar{B}_\ieps^4) \lesssim d_\ieps^2 \bar{B}_\ieps^4 \lesssim \mu_\eps^2 d_\ieps^{-2} \to 0 \]
and 
\[ \left| d_\ieps^2 \bar{a}_\ieps \frac{\bar{u}_\ieps}{r_\ieps(z_\ieps)} \right| \lesssim \frac{d_\ieps^{-\vartheta} \mu_\eps^{1/2}}{\|R_\ieps\|_{L^\infty(\mathsf{b}_\ieps)}}  \lesssim \frac{\mu_\eps^{1/2 - \vartheta}}{\|R_\ieps\|_{L^\infty(\mathsf{b}_\ieps)}} \]
uniformly on compacts of $\R^3 \setminus \{0\}$. 
If $\|R_\ieps\|_\infty >> \mu_\eps^{1/2 - \vartheta}$, then, using that still $d_\ieps \to 0$, $\bar{r}_{i,0} := \lim_{\eps \to 0} \bar{r}_\ieps$ satisfies 
\[ -\Delta \bar{r}_{i,0} = 0 \qquad \text{ on } \R^3 \setminus \{0\}. \]
 But by \eqref{r lesssim x vartheta}, $\bar{r}_{i,0}$ is bounded near $0$ and thus can be extended to a harmonic function on all of $\R^3$. A Taylor expansion together with \eqref{r lesssim x vartheta} now shows that $\bar{r}_{i,0} (0) = \nabla \bar{r}_{i,0}(0) = 0$. As in Case 1, we can now derive a contradiction. 

Thus we must have $\|R_\ieps\|_{L^\infty(\mathsf{b}_\ieps)} \lesssim \mu_\ieps^{\frac 12 - \vartheta}$, i.e. $r_\ieps(x) \lesssim \mu_\eps^{\frac{1}{2} - \vartheta} |x-x_\ieps|^{1 + \vartheta}$, also in this case. 

\textit{Case 2.b) $d_\ieps \sim 1$ as $\eps \to 0$. } In this case there is no need for a blow-up argument. Instead, we can simply bound, by the definition of $z_\ieps$, 
\[ \frac{|r_\ieps(x)|}{|x-x_\ieps|^{1+ \vartheta}} \lesssim \frac{|r_\ieps(z_\ieps)|}{d_\ieps^{1 + \vartheta}} \lesssim |r_\ieps(z_\ieps)| \lesssim \mu_\ieps^{1/2}, \]
where the last inequality simply comes from the bound $|u_\eps| \lesssim B_\ieps$ on $\mathsf{b}_\ieps$ and the observation that $d_\ieps \sim 1$ implies $B_\ieps(z_\ieps) \lesssim \mu_\eps^{1/2}$. Thus 
\[ |r_\ieps(x)| \lesssim \mu_\eps^{1/2} |x-x_\ieps|^{1+\vartheta} \leq \mu_\eps^{1/2- \vartheta} |x-x_\ieps|^{1+\vartheta}, \]
which completes the discussion of this case. 

\textit{Case 3. $\mu_\eps \sim d_\ieps$ as $\eps \to 0$. } This is the most delicate case because the right side of \eqref{bar r eps equation} now tends to a non-trivial limit. Indeed, $\beta_{i,0} := \lim_{\eps \to 0} \beta_\ieps := \lim_{\eps \to 0} \frac{\mu_\ieps}{d_\ieps}$ exists and $\beta_{i,0}  \in (0,\infty)$. Then 
\[ d_\ieps^{1/2} \bar{B}_\ieps = \frac{´\beta_\ieps^{1/2} }{(\beta_\ieps^2 + \frac{|x|^2}{3})^{1/2}} \to \frac{´\beta_{i,0}^{1/2} }{(\beta_{i,0}^2 + \frac{|x|^2}{3})^{1/2}} =: B_{0,\beta_{i,0}}. \]
By the convergence of $u_\eps$ from Proposition \ref{proposition preliminaries multi}, we also have $d_\ieps^{1/2} \bar{u}_\ieps \to B_{0,\beta_{i,0}}$ uniformly on compacts of $\R^3$. Moreover
\[ \left| d_\ieps^2 \bar{a}_\ieps \frac{\bar{u}_\ieps}{r_\ieps(z_\ieps)} \right| \lesssim \frac{d_\ieps^{\frac 12 -\vartheta} }{\|R_\ieps\|_{L^\infty(\mathsf{b}_\ieps)}}  \lesssim \frac{\mu_\eps^{\frac 12 - \vartheta}}{\|R_\ieps\|_{L^\infty(\mathsf{b}_\ieps)}} \]
If $\|R_\ieps\|_{L^\infty(\mathsf{b}_\ieps)} >> \mu_\ieps^{1/2 - \vartheta}$, we therefore recover the limit equation
\[  -\Delta \bar{r}_{i,0} = 5 \bar{r}_{i,0} B_{0, \beta_{i,0}}^4  \qquad \text{ on } \R^3, \]
which is precisely the linearized equation \eqref{lin eq}, up to a harmless rescaling. By \eqref{r lesssim x vartheta}, we have $|r_{i,0}(x)| \lesssim |x|^{1+ \vartheta}$ for all $x \in \R^3$. Thus by Corollary \ref{corollary u = 0} we conclude $\bar{r}_{i,0} \equiv 0$. This contradicts $\bar{r}_{i,0}(\xi_{i,0}) = 1$, as desired. 

Thus we  have shown $\|R_\ieps\|_{L^\infty(\mathsf{b}_\ieps)} \lesssim \mu_\eps^{\frac 12 - \vartheta}$, i.e. $r_\ieps(x) \lesssim \mu_\ieps^{\frac{1}{2} - \vartheta} |x-x_\ieps|^{1 + \vartheta}$, also in the third and final case. 
\end{proof}

\subsection{A refined expansion}
\label{subsection refined expansion}

With Proposition \ref{proposition r eps} at hand, we now complete the proof of Theorem \ref{theorem asmyptotic expansion additive} by proving the bound \eqref{q eps bound theorem} on 
\[ 	q_\ieps = u_\ieps - B_\ieps - \mu_\ieps^2 W_\ieps, \]
where $W_\ieps$ is defined in \eqref{W ieps definition}. Again, we restate the bound here for convenience. 

\begin{proposition}
\label{proposition q}
As $\eps \to 0$, for all $0 < \nu  <1$,
\[ |q_\ieps(x)| \lesssim \mu_\eps^{\frac{1}{2} - \nu} |x-x_\ieps|^{2 + \nu}, \qquad \text{ for all } x \in \mathsf{b}_\ieps. \]
\end{proposition}

\begin{proof}
We abbreviate $a_\eps = a + \eps V$. 
Using the definition of the $W_{jk}$ in \eqref{w equation} and of $W_\ieps$ in \eqref{W ieps definition}, we find 
\begin{align*}
-\Delta W_\ieps - 5 B_\ieps^4 W_\ieps &= \mu_\ieps^\frac{1}{2} B_\eps (\Delta u_\eps - \Delta B_\eps)(x_\ieps) \\
&= \mu_\ieps^\frac{1}{2} B_\eps a_\eps(x_\ieps) u(x_\ieps) = a_\eps(x_\ieps) B_\eps. 
\end{align*} 
With this it is easily checked that $q_\ieps$ satisfies the equation 
\begin{equation}
\label{q equation}
-\Delta q_\ieps + a_\eps r_\ieps + (a_\eps - a_\eps(x_\ieps)) B_\ieps  = 5 B_\ieps^4 q_\ieps + \mathcal O( r_\ieps^2 B_\ieps^3). 
\end{equation}
We now insert the bounds $|a_\eps| \lesssim 1$, and $|a_\eps(x) - a_\eps(x_\ieps)| \lesssim |x-x_\ieps|$ because $a, V \in C^1(\Omega)$. Together with the previously proved bounds $u_\eps \lesssim B_\ieps$ on $\mathsf{b}_\ieps$ from Proposition \ref{proposition preliminaries multi} and $r_\ieps \lesssim \mu_\eps^{1/2 - \nu} |x-x_\ieps|^{1 + \nu}$ on $\mathsf{b}_\ieps$ from Proposition \ref{proposition r eps}, we find that on $\mathsf{b}_\ieps$, 
\begin{equation}
\label{q equation bis}
\left| -\Delta q_\ieps - 5 q_\ieps B_\ieps^4 \right| \lesssim |x-x_\ieps|  B_\ieps +  \mu_\eps^{1/2 - \nu} |x-x_\ieps|^{1+\nu} + \mu_\eps^{1 - 2 \nu} |x-x_\ieps|^{2 + 2 \nu} B_\ieps^3 
\end{equation}
Now for some $0 < \nu < 1$, let 
\[ Q_\ieps(x) := \frac{q_\ieps(x)}{|x-x_\ieps|^{2+\nu}} \]
and denote by $z_\ieps \in \mathsf{b}_\ieps$ a point where $Q_\ieps(z_\ieps) \geq \frac 12 \|Q_\ieps\|_{L^\infty(\mathsf{b}_\ieps)}$ and set $d_\ieps := |z_\ieps - x_\ieps|$. 
(Notice that, by construction of $W_\ieps$, the function $q_\ieps$ satisfies $q_\ieps(x_\ieps) = \nabla q_\ieps(x_\ieps) = D^2 q_\ieps(x_\ieps) = 0$. Hence $\Vert Q_\ieps\|_{L^\infty(\mathsf{b}_\ieps)} < \infty$ by Taylor's theorem.)
We introduce the function 
\[ \bar{q}_\ieps(x) := \frac{q_\ieps(x_\ieps + d_\ieps x)}{q_\ieps(z_\ieps)}, \qquad   x \in B(0, d_\ieps^{-1} \delta_0), \] 
which satisfies 
\begin{equation}
\label{bar q lesssim x 1 + nu}
\bar{q}_\ieps(x) \lesssim |x|^{2 +\nu}.
\end{equation}
Moreover, multiplying \eqref{q equation bis} by $\frac{d_\ieps^2}{q_\ieps(z_\ieps)}$ and observing that $\frac{1}{q_\ieps(z_\ieps)} \leq 2 \frac{d_\ieps^{-2 - \nu}}{\|Q_\ieps\|_{L^\infty(\mathsf{b}_\ieps)}}$, 
\begin{equation}
\label{bar q equation}
\left| -\Delta \bar{q}_\ieps - 5 d_\ieps^2 \bar{q}_\ieps \bar{B}_\ieps^4 \right| \lesssim \frac{1}{\|Q_\ieps\|_{L^\infty(\mathsf{b}_\ieps)}} \left( d_\ieps^{1-\nu}  \bar{B}_\ieps +  \mu_\eps^{1/2 - \nu} d_\ieps + \mu_\eps^{1 - 2 \nu} d_\ieps^{2 + \nu} \bar{B}_\ieps^3  \right),
\end{equation}
locally on $\R^3$. 
Now we again distinguish three cases. Since the argument is analogous to that of Proposition \ref{proposition r eps}, we shall be a bit briefer here. 

\textit{Cases 1 and 2.a) $\mu_\eps >> d_\ieps$ or $\mu_\eps << d_\ieps << 1$ as $\eps \to 0$. } In these cases $\bar{B}_\ieps \lesssim \mu_\eps^{-1/2}$ and $\bar{B}_\ieps \lesssim d_\ieps^{-1} \mu_\eps^{1/2}$ respectively. In both cases one finds again that $d_\ieps^2 \bar{B}_\ieps^4 \to 0$. Moreover, using $0 < \nu < 1$  the right side of \eqref{bar q equation} can be bounded by 
\[ \frac{\mu_\eps^{\frac{1}{2} - \nu}}{\|Q_\ieps\|_{L^\infty(\mathsf{b}_\ieps)}}. \]
If $\|Q_\ieps\|_{L^\infty(\mathsf{b}_\ieps)} >> \mu_\eps^{\frac{1}{2} - \nu}$, then the limit $\bar{q}_{i,0}$ satisfies
\[ -\Delta \bar{q}_{i,0} = 0 \quad \text{ on } \R^3,  \qquad \bar{q}_{i,0}(x) \lesssim |x|^{2 + \nu} \quad \text{ on } \R^3. \]
This bound implies on the one hand
\[ \bar{q}_{i,0}(0) = 0, \quad \nabla \bar{q}_{i,0}(0) = 0, \quad D^2 \bar{q}_{i,0}(0) = 0, \]
and on the other hand
\[ \bar{q}_{i,0}(x) = a + b \cdot x + \langle x, C x \rangle \]
for some $a \in \R$, $b \in \R^3$, $C \in \R^{3 \times 3}$, where we can assume $C$ symmetric. It is easy to see that in fact $a = \bar{q}_{i,0}(0) = 0$, $b = \nabla \bar{q}_{i,0}(0)= 0$ and $C = \frac{1}{2} D^2 \bar{q}_{i,0}(0) = 0$. Hence $\bar{q}_{i,0} \equiv 0$. 
As before, this contradicts the fact that $\bar{q}_{i,0}(\xi_{i,0}) = 1$ for some $\xi_{i,0} \in \mathbb S^2$. Thus $\|Q_\ieps\|_{L^\infty(\mathsf{b}_\ieps)} \lesssim \mu_\eps^{\frac{1}{2} - \nu}$, which implies the proposition in Cases 1 and 2.a).

\textit{Case 2.b) $d_\ieps \sim 1$ as $\eps \to 0$.   } 

As above, the definition of $z_\ieps$ gives, on $\mathsf{b}_\ieps$, 
\[ \frac{|q_\ieps(x)|}{|x-x_\ieps|^{2+ \vartheta}} \lesssim \frac{|q_\ieps(z_\ieps)|}{d_\ieps^{2 + \vartheta}} \lesssim |q_\ieps(z_\ieps)| \lesssim \mu_\eps^{1/2}.  \]
The last inequality comes from the bound $|u_\eps| \lesssim B_\ieps$ on $\mathsf{b}_\ieps$ and the observation that $d_\ieps \sim 1$ implies $B_\ieps(z_\ieps) \lesssim \mu_\eps^{1/2}$. This gives the claimed bound in this case. 

\textit{Case 3. $\mu_\ieps \sim d_\ieps$ as $\eps \to 0$. } Let $\beta_{i,0} = \lim_{\eps \to 0} \beta_\ieps = \lim_{\eps \to 0} \frac{\mu_\ieps}{d_\ieps}$, then $d_\ieps^2 \bar{B}_\ieps^4 \to \frac{\beta_{i,0}^2}{(\beta_{i,0}^2 + |x|^2/3)^2} =: B_{0, \beta_{i,0}}^4$. If $\|Q_\ieps\|_{L^\infty(\mathsf{b}_\ieps)} >> \mu_\eps^{\frac{3}{2} - \nu}$, then by the same estimates on \eqref{bar q equation} as in Case 1, the right side of \eqref{bar q equation} tends to zero and the limit function $\bar{q}_{i,0}$ satisfies
\[ -\Delta \bar{q}_{i,0} = 5 \bar{q}_{i,0} B_{0, \beta_{i,0}}^4 \quad \text{ on } \R^3, \qquad  |\bar{q}_{i,0}(x)| \lesssim |x|^{2+\nu} \quad \text{ on } \R^3.  \]
Now we can invoke Corollary \ref{corollary u = 0} to obtain $\bar{q}_{i,0}(0) \equiv 0$, which contradicts $\bar{q}_{i,0}(\xi_{i,0}) =1$. Thus $\|Q_\ieps\|_{L^\infty(\mathsf{b}_\ieps)} \lesssim \mu_\eps^{\frac{1}{2} - \nu}$ also in this case, and the proof of the proposition is complete. 
\end{proof}

\section{Proof of Theorem \ref{theorem multibubble}}
\label{section proof of thm}

   \subsection{The main expansions}

By applying the expansion of $u_\eps$ near the concentration points derived in Theorem \ref{theorem asmyptotic expansion additive}, we can prove the following expansions. 

We will also need the matrix $\tilde{M}_a^l(\bm{x}) \in \R^{n \times n} = (\tilde{m}_{ij}^l(\bm{x}))_{i,j = 1}^n$ with entries
\begin{equation}
\label{m ij tilde definition}
\tilde{m}_{ij}^l(\bm{x}) := 
\begin{cases}
\partial_l \phi_a(x_i) & \text{ for } i = j, \\
- 2 \partial_l^x G_a(x_i, x_j) & \text{ for } i \neq j. 
\end{cases}
\end{equation}

Finally, recall that we have defined in Proposition \ref{proposition preliminaries multi} $\lambda_{i,0} = \lim_{\eps \to 0} \lambda_{\ieps} = \lim_{\eps \to 0} \frac{\mu_\ieps^{1/2}}{\mu_\oeps^{1/2}}$. Now define $\widetilde{\mathcal G}:= 4 \pi \sqrt 3 \sum_i \lambda_{i,0} G_a(x, x_i)$ and denote 
\[ Q_V(y):= \int_\Omega V(x) \widetilde{\mathcal G}(x) G_a(x, y) \diff x.   \]

Then the following expansions hold. 
 
\begin{proposition}
\label{proposition expansions}
Let $\bm{x}_\eps = (x_\oeps,...,x_{n,\eps}) \in \R^{3n}$ and $\bm{\lambda}_\eps = (\lambda_\oeps,...\lambda_{n,\eps})$ be as in Proposition \ref{proposition preliminaries multi}. 
As $\eps \to 0$, we have
\begin{equation}
\label{expansion G a line i}
 \eps (Q_V(x_\ieps) + o(1)) = - 4 \pi \sqrt 3 (M_a(\bm{x}_\eps) \cdot \bm{\lambda}_\eps)_i + 3 \pi (a(x_\ieps) +o(1)) \lambda_\ieps \mu_\ieps 
\end{equation}
and, for every $0 < \nu < 1$, 
\begin{equation}
\label{expansion nabla G a line i}
(\tilde{M}^l_a(\bm{x}_\eps) \cdot \bm{\lambda}_\eps)_i = \mathcal O(\eps + \mu_\eps^\nu),
\end{equation}
\end{proposition}
 
Before giving the proof of Proposition \ref{proposition expansions}, we observe the following property of the function $\widetilde{\mathcal G}$ defined at the beginning of this section. 

\begin{lemma}
\label{lemma tilde G}
As $\eps \to 0$, we have  $\mu_{1,\eps}^{-1/2} u_{\eps} \to   \widetilde{\mathcal G}$ uniformly away from $\{x_{1,0},...,x_{n,0}\}$.
\end{lemma}

\begin{remark}
\label{remark G = tilde G}
Note that $\widetilde{\mathcal G}$ is defined in terms of the $\lambda_{i,0}$ from Proposition \ref{proposition preliminaries multi}, while the function $\mathcal G$ appearing in Theorem \ref{theorem multibubble} is defined in terms of the eigenvector $\bm{\Lambda}_0$. We shall however prove in Lemma \ref{lemma perron frobenius} below that $\bm{\lambda}_0 = \bm{\Lambda}_0$ and hence in fact $\widetilde{\mathcal G} = \mathcal G$.
\end{remark}

\begin{proof}
[Proof of Lemma \ref{lemma tilde G}]
By applying $(-\Delta + a)^{-1}$ to the equation satisfied by $u_\eps$, we obtain
\[ u_\eps (x) = \int_\Omega G_a(x,y) (u_\eps^5(y) + \eps V(y) u(y)) \diff y \]
for every $x \in \Omega$. By developing $u_\eps = B_\ieps + r_\ieps$ near $x_\ieps$ and using $|r_\ieps(x)| \lesssim \mu_\eps^{1/2 -\vartheta}|x-x_\ieps|^{1+\vartheta}$ for every $\vartheta \in (0,1)$ by Theorem \ref{theorem asmyptotic expansion additive}, we get 
\begin{align*}
\int_\Omega G_a(x,y) u_\eps^5(y) \diff y &= \sum_i \left( \int_{\mathsf b_\ieps} G_a(x,y) u_\eps^5(y) \diff y \right) + \int_{\Omega \setminus \bigcup_i \mathsf{b}_\ieps} G_a(x,y) u_\eps^5(y) \diff y \\
&=  \sum_i \mu_\ieps^{1/2} (\int_{\R^3} B^5 \diff y) G_a(x, x_\ieps) + o(\mu_\eps^{1/2}) 
\end{align*}
uniformly for $x$ in compacts of $\Omega \setminus \{x_{1,0},...,x_{n,0}\}$.
Moreover, the bound $u_\eps \lesssim B_\ieps$ near $x_\ieps$ from Proposition \ref{proposition expansions} easily gives 
\[ \eps  \int_\Omega G_a(x,y)  V(y) u(y)\diff y  =  \mathcal O(\eps \mu_\eps^{1/2}) = o(\mu_\eps^{1/2}), \]
uniformly for $x \in \Omega$. 

Since $\int_{\R^3} B^5 \diff y = 4 \pi \sqrt 3$, combining all of the above, dividing by $\mu_\oeps^{1/2}$ and recalling the definition of $\lambda_{i,0}$ gives the conclusion.  
\end{proof}

\begin{proof}
[Proof of Proposition \ref{proposition expansions}]

\textit{Proof of \eqref{expansion G a line i}.   }
Integrate equation \eqref{brezis peletier additive} for $u_\eps$ against $G_a(x_\ieps, \cdot)$ to get 
\begin{equation}
\label{integral identity prop proof}
\int_\Omega (-\Delta + a) u_\eps G_a(x, x_\ieps) \diff x  + \eps  \int_\Omega V u_\eps G_a(x_\ieps, \cdot) \diff x = \int_\Omega u_\eps^5 G_a(x_\ieps, \cdot) \diff x 
\end{equation} 
Then by the definition of $G_a(x_\ieps, \cdot)$ and by the convergence $\mu_\oeps^{-1/2} u_\eps \to  \widetilde{ \mathcal G}$ from Lemma \ref{lemma tilde G}, the left side of \eqref{integral identity prop proof} equals
\begin{align*}
 u_\eps(x_\ieps) + \eps  \mu_\eps^{1/2} \lambda_\ieps^{-1}  (Q_V(x_\ieps) + o(1)) = \mu_\ieps^{-1/2} + \eps  \mu_\eps^{1/2} \lambda_\ieps^{-1}  (Q_V(x_\ieps) + o(1)).
\end{align*}
Evaluating the right side of \eqref{integral identity prop proof} requires some more care. We start by writing
\begin{align*}
\int_\Omega u_\eps^5 G_a(x_\ieps, \cdot) \diff x = \sum_j \int_{\mathsf{b}_\jeps}  u_\eps^5 G_a(x_\ieps, \cdot) \diff x + \int_{\Omega \setminus \bigcup_j \mathsf{b}_\jeps}   u_\eps^5 G_a(x_\ieps, \cdot) \diff x. 
\end{align*}
Clearly, on $\Omega \setminus \bigcup_i \mathsf{b}_\ieps$ we have $u_\eps \lesssim \sum B_\ieps \lesssim \mu_\eps^{1/2}$, and so 
\[  \int_{\Omega \setminus \bigcup_j \mathsf{b}_\jeps}   u_\eps^5 G_a(x_\ieps, \cdot) \diff x \lesssim \mu_\eps^{5/2} = o(\mu_\eps^{3/2}). \]
It therefore remains to evaluate the integral over the balls $\mathsf{b}_\jeps$, up to $o(\mu_\eps^{3/2} + \eps \mu_\eps^{1/2})$ precision. We shall consider the cases $j = i$ and $j \neq i$ separately. 

\textit{Case $j = i$.  }

Careful, but straightforward computations and estimates give 
\[ \int_{\mathsf{b}_\ieps} B_\ieps^5 G_a(x_\ieps, \cdot) \diff x = \mu_\ieps^{-1/2} - 4 \pi \sqrt 3 \phi_a(x_\ieps) \mu_\ieps^{1/2} + 3  a(x_\ieps) \mu_\ieps^{3/2} + o(\mu_\eps^{3/2}), \]
see Lemma \ref{lemma int B^5 Ga} below.

To evaluate the error term on the right side, we need the full precision of the asymptotic expansion of $u_\eps$ derived in Theorem \ref{theorem asmyptotic expansion additive}. By that theorem, on $\mathsf{b}_\ieps$ we may write 
\[ u_\eps^5 - B_\ieps^5  =  5 \mu_\ieps^2 W_\ieps B_\ieps^4 +  \mathcal{O}(|q_\ieps|B_\ieps^4 + r_\ieps^2 B_\ieps^3),  \]  
with the remainders $r_\ieps$ and $q_\ieps$ satisfying the bounds $|r_\ieps| \lesssim \mu_\eps^{\frac 12  - \vartheta} |x-x_\ieps|^{1+\vartheta}$ and $|q_\ieps| \lesssim \mu_\eps^{\frac 12 - \nu} |x-x_\ieps|^{2+ \nu}$ on ${\mathsf{b}_\ieps}$ respectively (with $0 < \nu, \vartheta < 1$). Thus we get  
\begin{align*}
& \quad  \int_{\mathsf{b}_\ieps} \left(|r_\ieps| H_a(x_\ieps, \cdot) + \frac{q_\ieps}{|x-x_\ieps|} \right) B_\ieps^4 \diff x +  \int_{\mathsf{b}_\ieps} |u_\eps - B_\ieps|^2 B_\ieps^3 G_a(x_\ieps, \cdot) \diff x \\
&\lesssim \mu_\eps^{\frac{5}{2} - \vartheta} + \mu_\eps^{\frac{5}{2} - \nu} + \mu_\eps^{\frac{5}{2}- 2 \vartheta} = o(\mu_\eps^{3/2}). 
\end{align*}

Thus 
\begin{align}
 \int_{\mathsf{b}_\ieps} (u_\eps^5 - B_\ieps^5) G_a(x_\ieps, \cdot) \diff x   &= \frac{5}{4 \pi} \mu_\ieps^{2} \int_{\mathsf{b}_\ieps} W_\ieps |x-x_\ieps|^{-1} B_\ieps^4 \diff x + o(\mu_\eps^{3/2}). \label{ref full estimate}
\end{align}
Recall that $W_\ieps(x) = \sum_{j,k} c_{jk,\eps} W_{jk}(\frac{x-x_\ieps}{\mu_\ieps})$ with $c_{jk, \eps} = \partial_{jk}(u_\eps - B_\ieps)(x_\ieps)$. Now if $j \neq k$, then $W_{jk}(x) = f(x) Y_2(x/|x|)$ for some spherical harmonic $Y_2$ of degree 2. Hence its integral over the ball $\mathsf{b}_\ieps$ against the radial function $|x-x_\ieps|^{-1} B_\ieps^4$ vanishes. Thus only the terms with $j = k$ remain, and for those we have 
\begin{align*}
 \sum_j W_{jj}\left(\frac{x - x_\ieps}{\mu_\ieps}\right) &= W\left(\frac{x - x_\ieps}{\mu_\ieps}\right) \Delta (u_\eps - B_\ieps)(x_\ieps) \\
 &= \left(a(x_\ieps) + \eps V(x_\ieps)\right) \mu_\ieps^{-\frac{1}{2}} W\left(\frac{x - x_\ieps}{\mu_\ieps}\right),
\end{align*} 
where $W$ is the function from Lemma \ref{lemma integral w0}. 
It follows that 
\begin{align*}
&\quad  \int_{\mathsf{b}_\ieps} (u_\eps^5 - B_\ieps^5) G_a(x_\ieps, \cdot) \diff x \\
&= \mu_\ieps^{3/2} \frac{5}{ 4 \pi} (a(x_\ieps) + \eps V(x_\ieps)) \int_{\R^3} \frac{W(x) B(x)^4}{|x|} \diff x + o(\mu_\eps^{3/2})  \\
&= 3 a(x_\ieps) (\pi - 1) \mu_\ieps^{3/2} + o(\mu_\eps^{3/2}),
\end{align*}
where the final identity is computed in Lemma \ref{lemma integral w0}. Collecting all expansions, we conclude the proof of \eqref{expansion G a line i}. 

\textit{Case $j \neq i$.  } The expansion of the cross terms with $j \neq i$ is simpler because $G_a(x_\ieps, \cdot)$ is bounded on $\mathsf{b}_\jeps$ in that case, and we give a bit fewer details.  Using the bound \eqref{r eps bound theorem} on $r_\ieps$ from Theorem \ref{theorem asmyptotic expansion additive} we write 
\begin{align*}
&\qquad \int_{\mathsf{b}_\jeps}  u_\eps^5 G_a(x_\ieps, \cdot) \diff x \\
&= \int_{\mathsf{b}_\jeps}  B_\jeps^5 G_a(x_\ieps, \cdot) \diff x + \mathcal O\left( \mu_\eps^{1/2 - \vartheta} \int_{\mathsf{b}_\jeps}  B_\jeps^4 |x - x_\jeps|^{1 + \vartheta}  \diff x \right) \\
&= 4 \pi \sqrt 3 \mu_\jeps^{1/2} G_a(x_\ieps, x_\jeps) + o(\mu_\eps^{3/2}).  
\end{align*} 
In summary, inserting everything into \eqref{integral identity prop proof}, we have proved
\begin{align*}
&\qquad \mu_\ieps^{-1/2} + \eps  \mu_\ieps^{1/2} \lambda_\ieps^{-1}  (Q_V(x_\ieps) + o(1))\\
& = 4 \pi \sqrt{3} \left( -\phi_a(x_\ieps) + G_a(x_\ieps, x_\jeps) \right) + 3 (a(x_\ieps) +o(1)) \mu_\ieps^{3/2}.  
\end{align*} 
Dividing by $\mu_\oeps^{1/2}$ and recalling the definitions of $\lambda_\ieps$ and $M_a(\bm{x}_\eps)$, we obtain \eqref{expansion G a line i}.

\textit{Proof of \eqref{expansion nabla G a line i}.   }

We proceed similarly to the proof of expansion \eqref{expansion G a line i} and multiply equation \eqref{brezis peletier additive} for $u_\eps$ against $\nabla_x G_{a+\eps V}(x_\ieps, \cdot)$.

 Then 
\begin{align*}
&\qquad \int_\Omega (-\Delta + a + \eps V) u_\eps \nabla_x G_{a + \eps V}(x_\ieps, \cdot) \diff y \\
 &= \nabla_x \int_\Omega (-\Delta + a + \eps V) u_\eps  G_{a + \eps V}(x_\ieps, \cdot) \diff y \\
&= \nabla u_\eps(x_\ieps) = 0. 
\end{align*}

The right side of equation \eqref{brezis peletier additive} integrated against $\nabla_x G_{a + \eps V}(x_\ieps, \cdot)$ can be written as 
\begin{align*}
&\qquad \int_\Omega u_\eps^5 \nabla_x G_{a+\eps V}(x_\ieps, \cdot) \diff x \\
& = \sum_j \left( \int_{\mathsf{b}_\jeps}  u_\eps^5 \nabla_x G_a(x_\ieps, \cdot) \diff x \right) +  \int_{\Omega \setminus \bigcup_j \mathsf{b}_\jeps}   u_\eps^5 \nabla_x G_a(x_\ieps, \cdot) \diff x. 
\end{align*}
The last term on the right side, since $u_\eps  \lesssim \sum_k B_\keps \lesssim \mu_\eps^{1/2}$ on $\Omega \setminus \bigcup_j \mathsf{b}_\jeps$, is bounded by 
\[ \int_{\Omega \setminus \bigcup_j \mathsf{b}_\jeps}   u_\eps^5 \nabla_x G_a(x_\ieps, \cdot) \diff x \lesssim \mu_\eps^{5/2}. \]
It remains to evaluate the integral over the balls $\mathsf{b}_\jeps$. We again consider the cases $j = i$ and $j \neq i$ separately. 

\textit{Case $j = i$. }
We write
\begin{align}
&\qquad \int_{\mathsf{b}_\ieps} u_\eps^5 \nabla_x G_{a+\eps V}(x_\ieps, \cdot) \diff y \\
&= \int_{\mathsf{b}_\ieps} B_\ieps^5 \nabla_x G_{a+\eps V}(x_\ieps, \cdot) \diff x + 5 \int_{\mathsf{b}_\ieps} r_\ieps B_\ieps^4 \nabla_x G_{a+\eps V}(x_\ieps, \cdot) \diff x \nonumber \\
&\qquad + \mathcal O\left( \int_{\mathsf{b}_\ieps} r_\ieps^2 \frac{B_\ieps^3}{|x-x_\ieps|^2} \diff x  \right). \label{bigO proof}
\end{align}
Let us treat the terms on the right side one by one. The first term, by explicit computations carried out in Lemma \ref{lemma int B^5 nabla Ga}, is 
\begin{align*} \int_{\mathsf{b}_\ieps} B_\ieps^5 \nabla_x G_{a + \eps V}(x_\ieps,\cdot) \diff y &=  - 2 \pi \sqrt 3 \nabla \phi_{a + \eps V}(x_\ieps) \mu_\ieps^{1/2} + o(\mu_\eps^{\frac 12  + \nu}) \\
&= - 2 \pi \sqrt 3 \nabla \phi_{a}(x_\ieps) \mu_\ieps^{1/2} + \mathcal O(\eps \mu_\eps^{1/2}) + o(\mu_\eps^{\frac 12  + \nu})
\end{align*}
for every $0 < \nu < 1$. The last equality comes from the fact that by the resolvent formula, 
\begin{align*}
\nabla \phi_{a + \eps V}(x_\ieps) &= \nabla \phi_a(x_\ieps) + \eps \nabla_x  \int_\Omega G_a(x, y) V(y) G_{a + \eps V}(x, y) \diff y \, \Big|_{x = x_\ieps} \\
&= \nabla \phi_a(x_\ieps) + \mathcal O(\eps), 
\end{align*} 
see Lemma \ref{lemma nabla phi a + eps V} for details.

To estimate the last term on the right side of \eqref{bigO proof}, the bound \eqref{r eps bound theorem} on $r_\ieps$ from Theorem \ref{theorem asmyptotic expansion additive} gives
\begin{align*}
&\qquad  \int_{\mathsf{b}_\ieps} r_\ieps^2 \frac{B_\ieps^3}{|x-x_\ieps|^2} \diff x  \lesssim \mu_\eps^{1 - 2 \vartheta} \int_\Omega B_\ieps^3 |x-x_\ieps|^{2 \vartheta} \diff x 
 \lesssim \mu_\eps^{\frac{5}{2} - 2 \vartheta} . 
\end{align*}

Finally, let us show that the second term on the right side of \eqref{bigO proof} is negligible. Here is where we use the full strength of Theorem \ref{theorem asmyptotic expansion additive}, i.e. the bound \eqref{q eps bound theorem} on $q_\ieps$, once more after using it to get \eqref{ref full estimate} in the proof of \eqref{expansion G a line i}. We write
\begin{align*}
& \quad \int_{\mathsf{b}_\ieps} r_\ieps B_\ieps^4 \nabla_y G_{a+\eps V}(x_\ieps, \cdot) \diff y \\
& = - \frac{1}{4\pi} \int_{\mathsf{b}_\ieps} (\mu_\ieps^2 W_\ieps + q_\ieps) B_\eps^4 \frac{x_\eps - x}{|x-x_\eps|^3} \diff x + \int_{\mathsf{b}_\ieps} (u_\eps - B_\eps)B_\eps^4 \nabla_x H_{a+\eps V}(x_\eps, \cdot) \diff x. 
\end{align*} 
Recall that $W_\ieps(x) = \sum_{j,k} c_{jk,\eps} W_{jk}(\frac{x-x_\ieps}{\mu_\ieps})$, where all the functions $W_{jk}$ are of the form $W_{jk}(x) = f(|x|) Y(x/|x|)$ for some $Y$ which is a sum of spherical harmonics of degree $0$ and $2$. Since on the other hand the vector $\frac{x_\ieps - x}{|x_\ieps - x|}$ is made of spherical harmonics of degree $1$, the integral of $W_\ieps$ against $B_\ieps^4 \frac{x_\ieps - x}{|x-x_\ieps|^3}$ vanishes. 
%\sout{Since $W_\ieps$ is radial with respect to $x_\ieps$, the integral of $W_\ieps B_\ieps^4 \frac{x_\ieps - x}{|x_\ieps - x|^3}$ over the ball ${\mathsf{b}_\ieps}$ vanishes!} 
Moreover, by \eqref{q eps bound theorem},
\begin{align*}
\int_{\mathsf{b}_\ieps} |q_\ieps|  B_\ieps^4 \frac{x_\ieps - x}{|x-x_\ieps|^3} \diff x \lesssim \mu_\eps^{\frac 12 - \nu} \int_{\R^3} |x - x_\ieps|^{ \nu} B_\ieps^4 \diff x \lesssim  \mu_\eps^{3/2}. 
\end{align*} 

To control the remaining term, by \eqref{r eps bound theorem} and since $\nabla_x H_{a+\eps V}(x_\ieps, \cdot)$ is uniformly bounded on $\Omega$, we have
\begin{align*}
\left|  \int_{\mathsf{b}_\ieps} (u_\eps - B_\ieps) B_\ieps^4 \nabla_x H_{a+\eps V}(x_\ieps, \cdot) \diff x \right| &\lesssim \mu_\eps^{\frac{1}{2} - \vartheta} \int_\Omega B_\ieps^4 |x-x_\ieps|^{1+ \vartheta} \diff x \lesssim \mu_\eps^{\frac 52 -\vartheta} . 
\end{align*}

\textit{Case $j \neq i$.  }

Again, this case is less involved and we will be briefer. We have, using the bound on $|u_\eps - B_\jeps|$ from Proposition \ref{proposition r eps}, 
\begin{align*}
& \qquad \int_{\mathsf{b}_\jeps} u_\eps^5 \nabla_x G_{a+\eps V}(x_\ieps, \cdot) \diff x \\
 &= \int_{\mathsf{b}_\jeps} B_\jeps^5 \nabla_x G_{a+\eps V}(x_\ieps, \cdot) \diff x  \\
 & \qquad + \mathcal O\left( \mu_\eps^{\frac 12 - \vartheta} \int_{\mathsf{b}_\jeps}  B_\jeps^4 |x-x_\jeps|^{1 + \vartheta} \nabla_x G_{a+\eps V}(x_\ieps, \cdot) \diff x \right) \\
 &= 4 \pi \sqrt 3 \mu_\jeps^{1/2} \nabla_x G_a(x_\ieps, x_\jeps) + o(\mu_\eps^{3/2}). 
\end{align*}
Now collecting all the estimates, choosing $\vartheta > 0$ small enough and dividing by $\mu_\oeps^{1/2}$ gives \eqref{expansion nabla G a line i}
\end{proof}

\begin{remark}
In the preceding proof, the choice of multiplying against $G_{a + \eps V}$ rather than $G_a$ is made on technical grounds. Indeed, by doing so, one does not need to bound $\int \eps V \nabla G_a u_\eps  \lesssim \eps \mu_\eps^{1/2}$ on the left side, but instead needs to bound $\nabla \phi_{a+ \eps V}(x_\ieps) - \nabla \phi_a(x_\ieps)$ on the right side. The latter can be better handled with the bounds we have so far. In fact, plugging in the (slightly non-optimal) bounds on $r_\ieps$ (or $q_\ieps$) into the first term leads to an estimate of the type $\int \eps V \nabla G_a u_\eps \lesssim \eps \mu_\eps^{1/2} \lesssim \eps^2 \mu_\eps^{1/2 - \nu}$. To get the desired bound, we would need $\eps \lesssim \mu_\eps^\nu$, which is unclear at this stage. For the bound $\nabla \phi_{a+ \eps V}(x_\ieps) - \nabla \phi_a(x_\ieps)$, proved in Lemma \ref{lemma nabla phi a + eps V}, a similar issue does not arise.
\end{remark}

\begin{remark}
\label{remark weaker bound on q}
Arguing as in Section \ref{section asymp analysis additive}, one can deduce, for all $0 < \nu < 1$, the bound $|q_\ieps(x)| \lesssim \mu_\eps^{1/2} |x - x_\ieps|^{1 + \nu}$ on $\mathsf{b}_\ieps$, which is weaker than \eqref{q eps bound theorem} near $x_\ieps$. This bound can however be checked to be just enough to appropriately bound the terms in $q_\ieps$ in the previous proof of Proposition \ref{proposition expansions}. 
\end{remark}

\subsection{Properties of the matrix $M_a(\bm{x})$ and the eigenvector $\rho_a(\bm{x})$}

For every $\bm{x} \in \Omega_\ast^n$, recall that we denote by $\rho_a(\bm{x})$ the lowest eigenvalue of the matrix $M_a(\bm{x})$. 

Moreover, denote by $\bm{x}_0 = (x_{1,0},...,x_{n,0})$ and $\bm{\lambda}_0$ the limit points of  $\bm{x}_\eps$ and $\bm{\lambda}_\eps$ respectively. (The vector $\bm{\lambda}_\eps$ is defined in Proposition \ref{proposition preliminaries multi}.) 

\begin{lemma}
\label{lemma perron frobenius}
\begin{enumerate}[(i)]
\item \label{item perron x} For every $\bm{x} \in \Omega_\ast^n$, $\rho_a(\bm{x})$ is a simple eigenvalue. The associated eigenvector can be chosen so that all of its entries are strictly positive. All other eigenvectors of $M_a(\bm{x})$ have both strictly negative and strictly positive entries. 
\item \label{item perron x0} For $\bm{x} = \bm{x}_0$, we have $\rho_a(\bm{x}_0) = 0$ with eigenvector $\bm{\Lambda} (\bm{x}_0) = \bm{\lambda}_0$. Moreover, $\nabla \rho_a(\bm{x}_0) = 0$. 
\end{enumerate}
\end{lemma}

In the following, for $\eps \geq 0$ let us abbreviate $\bm{\Lambda_\eps} := \bm{\Lambda}(\bm{x}_\eps)$. 

\begin{proof}
Assertion \eqref{item perron x} follows by the Perron--Frobenius argument detailed for the case $a = 0$ in \cite[Appendix A]{Bahri1995}. This argument still applies because it only relies on the strict negativity of all off-diagonal entries $-G_a(x_i,x_j)$, which is fulfilled in our case. 

For assertion \eqref{item perron x0}, expansion \eqref{expansion G a line i} plainly gives $M_a(\bm{x}_0) \cdot \bm{\lambda}_0 = 0$ by passing to the limit, hence $\bm{\lambda}_0$ is an eigenvector with eigenvalue $0$. Since  $\bm{\lambda}_0$ has strictly positive entries, it must be the lowest one by part \eqref{item perron x} of the lemma.  

It remains to prove that $\nabla \rho_a(\bm{x}_0) = 0$. To see this, note on the one hand that a direct calculation gives
\[ \partial_l^{x_i} \langle \bm{\lambda}_\eps, M_a(\bm{x}) \cdot \bm{\lambda}_\eps \rangle |_{\bm{x} = \bm{x_\eps}} = \lambda_\ieps (\tilde{M}^l_a(\bm{x}_\eps) \cdot \bm{\lambda}_\eps)_i  \]
for every $l=1,2,3$, $i=1,...,n$ and $\eps > 0$. Hence by expansion \eqref{expansion nabla G a line i}
\[ \partial_l^{x_i} \langle \bm{\lambda}_\eps, M_a(\bm{x}) \cdot \bm{\lambda}_\eps \rangle |_{\bm{x} = \bm{x_\eps}} = \mathcal O(\eps + \mu_\eps^\nu) \]
for every $0 < \nu < 1$. On the other hand, we can evaluate the same quantity as follows: Decompose, for $\bm{x}$ close to $\bm{x}_\eps$, the vector $\bm{\lambda}_\eps = \bm{\sigma}_\eps(\bm{x}) + \bm{\delta}_\eps(\bm{x})$, where $ \bm{\sigma}_\eps(\bm{x}) || \bm{\Lambda}(\bm{x})$ and $\bm{\delta}_\eps(\bm{x}) \perp  \bm{\Lambda}(\bm{x})$. Then we can express
\[ \langle \bm{\lambda}_\eps, M_a(\bm{x}) \cdot \bm{\lambda}_\eps \rangle = \rho_a(\bm{x}) |\bm{\sigma}_\eps(\bm{x})|^2 +  \langle \bm{\delta}_\eps(\bm{x}), M_a(\bm{x}) \cdot \bm{\delta}_\eps(\bm{x}) \rangle . \]
Using the fact that the dependence of all quantities on $\bm{x}$ is $C^1$, we obtain
\[ \partial_l^{x_i} \langle \bm{\lambda}_\eps, M_a(\bm{x}) \cdot \bm{\lambda}_\eps \rangle |_{\bm{x} = \bm{x_\eps}} = \partial_l^{x_i} \rho_a(\bm{x}) + \mathcal O(|\rho_a(\bm{x})|  + | \bm{\delta}_\eps(\bm{x})|).  \]
Combining these estimates, we get 
\begin{equation}
\label{nabla rho a a priori bound}
|\nabla \rho_a(\bm{x}_\eps)| = \mathcal O( \eps + \mu_\eps^\nu + |\rho_a(\bm{x}_\eps)|  + | \bm{\delta}_\eps(\bm{x}_\eps)|). 
\end{equation} 
Now the facts that $\rho_a(\bm{x}_0) = 0$ and $\bm{\lambda}_\eps \to \bm{\lambda}_0$ as $\eps \to 0$ clearly imply $|\rho_a(\bm{x}_\eps)|  + | \bm{\delta}_\eps(\bm{x}_\eps)| = o(1)$, which gives the conclusion.
\end{proof}

In the following, we will decompose the vector $\bm{\lambda}_\eps$ as 
\begin{equation}
\label{lambda = sigma + delta}
\bm{\lambda}_\eps = \bm{\sigma}_\eps + \bm{\delta}_\eps, 
\end{equation} 
where $\bm{\sigma}_\eps || \bm{\Lambda}_\eps$ and $\bm{\delta}_\eps \perp  \bm{\Lambda}_\eps$.  

\begin{lemma}
\label{lemma rho bounds quantitative}

As $\eps \to 0$, we have
\begin{equation}
\label{delta lesssim eps}
| \bm{\delta}_\eps| \lesssim \eps + \mu_\eps + |\rho_a(\bm{x}_\eps)|, 
\end{equation} 
and, for all $0 < \nu < 1$, 
\begin{equation}
\label{nabla rho lesssim eps nu}
|\nabla \rho_a(\bm{x}_\eps)| \lesssim \eps + \mu_\eps^\nu +  |\rho_a(\bm{x}_\eps)|.  
\end{equation} 
\end{lemma}

\begin{proof}
Writing $M_a(\bm{x}_\eps) \cdot \bm{\lambda}_\eps = \rho_a(\bm{x}_\eps) + M_a(\bm{x}_\eps) \cdot \bm{\delta}_\eps$, from \eqref{expansion G a line i} we plainly get
\[ | M_a(\bm{x}_\eps) \cdot \bm{\delta}_\eps | \lesssim \eps + \mu_\eps + |\rho_a(\bm{x}_\eps)|.  \]
Since $M_a(\bm{x}_\eps)$ has bounded inverse (independently of $\eps$) on the spectral subspace containing $\bm{\delta}_\eps$, this gives \eqref{delta lesssim eps}.

The bound \eqref{nabla rho lesssim eps nu} simply follows by inserting  \eqref{delta lesssim eps} into the a priori bound \eqref{nabla rho a a priori bound} which was already obtained in the proof of Lemma \ref{lemma perron frobenius}. 
\end{proof}

\subsection{Proof of Theorem \ref{theorem multibubble}}

Part \eqref{item conc points} of Theorem \ref{theorem multibubble} is contained in Proposition \ref{proposition preliminaries multi} and Lemma \ref{lemma perron frobenius}, and part \eqref{item global} in Lemmas \ref{lemma tilde G} and \ref{lemma perron frobenius}. Part \eqref{item local} is precisely the statement of Theorem \ref{theorem asmyptotic expansion additive}. So it only remains to prove part \eqref{item blowup} of Theorem \ref{theorem multibubble}. 

To prove \eqref{item blowup}, the crucial step is to  realize that both assumptions (a) and (b) on $\rho_a$ imply 
\begin{equation}
\label{rho a = o(eps)}
\rho_a(\bm{x}_\eps) = o(\eps + \mu_\eps). 
\end{equation}

Indeed, let us assume first that assumption (a) holds, that is, $D^2 \rho_a(\bm{x}_0) \geq c$
for some $c > 0$, in the sense of quadratic forms.  
Then \cite[Lemma 4.2]{Frank2021b} applies to give that  \[ \rho_a(\bm{x}_\eps) \lesssim |\nabla \rho_a(\bm{x}_\eps)|^2. \]
Using \eqref{nabla rho lesssim eps nu} with $\nu > 1/2$, and absorbing $\rho_a(\bm{x}_\eps)^2 = o(\rho_a(\bm{x}_\eps))$ on the left side, \eqref{rho a = o(eps)} follows. 

On the other hand, assume that (b) holds, that is, $\rho_a$ is real-analytic.

Up to extracting a subsequence, we may assume that $\frac{\bm{x}_\eps - \bm{x}_0}{|\bm{x}_\eps - \bm{x}_0|} \to \eta \in \mathbb S^{3n-1}$. (If $x_\eps = x_0$, then \eqref{rho a = o(eps)} is trivially true.) Let 
\[ f(t):= \rho_a(\bm{x}_0 + t \eta).  \]
Then $f$ is well-defined on some open interval containing $0$ and $f$ is analytic because $\rho_a$ is. Moreover, since $|\rho_a(\bm{x})| \to \infty$ as $\bm{x} \to \partial \Omega_*^n$, we must have $f \not \equiv 0$. As a consequence, there is a smallest  $k \geq 2$ such that $f^{(k)}(0) \neq 0$. Defining 
\[ f_\eps (t) := \rho_a \left(\bm{x}_0 + t \frac{\bm{x}_\eps - \bm{x}_0}{|\bm{x}_\eps - \bm{x}_0|} \right), \]
we clearly have $f_\eps^{(k)}(0) \to f^{(k)}(0)$ as $\eps \to 0$ and thus $f_\eps^{(k)}(0) \neq 0$ for all $\eps$ small enough. Let us assume for definiteness that $f_\eps^{(k)}(0) > 0$.  By analyticity and the choice of $k$, we may write 
\[ f_\eps(t) = \frac{f_\eps^{(k)}(0) +o(1)}{k!} t^k \]
and
\[ f'_\eps(t) = \frac{f_\eps^{(k)}(0) + o(1)}{(k-1)!} t^{k-1}. \]
Here $o(1)$ denotes a quantity that tends to zero as $t \to 0$, uniformly in $\eps$. Solving the second equation for $t$ and inserting it into the first, we obtain 
\[ f_\eps(t) = (c_k + o(1)) f'_\eps(t)^\frac{k}{k-1}, \qquad \text{ with } c_k = \frac{1}{k} \left( \frac{(k-1)!}{f^{(k)}(0)} \right)^\frac{1}{k-1} > 0.  \] 
Taking $t = t_\eps := |\bm{x}_\eps - \bm{x}_0|$, we get 
\[ |\rho_a(\bm{x}_\eps)| = |f_\eps(t_\eps)| \leq (c_k + o(1))  f'_\eps(t)^\frac{k}{k-1} \lesssim |\nabla \rho_a(\bm{x_\eps}) \cdot \frac{\bm{x}_\eps - \bm{x}_0}{|\bm{x}_\eps - \bm{x}_0|}| \lesssim  |\nabla \rho_a(\bm{x_\eps})|. \]
Now by using \eqref{nabla rho lesssim eps nu} with  $\nu < 1$ so large that $\nu \frac{k}{k-1} > 1$ and again absorbing $\rho_a(\bm{x}_\eps)^{\nu \frac{k}{k-1}} = o(\rho_a(\bm{x}_\eps))$, \eqref{rho a = o(eps)} follows also under assumption (b).

Armed with \eqref{rho a = o(eps)}, it is now straightforward to conclude the proof of Theorem \ref{theorem multibubble}.\eqref{item blowup}. Since $\ds \lim_{\eps \to 0} \eps \mu_{\jeps}^{-1} = \lambda_{j,0}^{-2} \ds \lim_{\eps \to 0} \eps \mu_{\oeps}^{-1}$, we only need to evaluate the limit $\ds \lim_{\eps \to 0} \eps \mu_{\oeps}^{-1}$. Using \eqref{lambda = sigma + delta} together with \eqref{rho a = o(eps)} and \eqref{delta lesssim eps}, identity \eqref{expansion G a line i} becomes
\begin{equation}
\label{identity prefinal}
\eps \lambda_\ieps (Q_V(x_\ieps) + o(1)) = - 4 \pi \sqrt 3 \sigma_\ieps (M_a(\bm{x}_\eps) \cdot \bm{\delta}_\eps)_i + 3 \pi (a(x_\ieps) +o(1)) \lambda_\ieps^2 \mu_\ieps . 
\end{equation} 
Now write $\mu_\ieps = \lambda_\ieps^2 \mu_\oeps$ and sum over $i$. Since $\bm{\sigma}_\eps \perp \bm{\delta}_\eps$, the first term on the right side of \eqref{identity prefinal} vanishes in the sum. Moreover, recalling that $\mathcal G = \tilde{\mathcal G}$ from Remark \ref{remark G = tilde G}, we can write  $4 \pi \sqrt 3 \sum_i Q_V(x_\ieps)  = \int_\Omega V(x) \mathcal G(x)^2 \diff x$. Thus we obtain 
\begin{equation}
\label{final identity}
 3 \pi \mu_\oeps  \sum_i (a(x_\ieps) +o(1)) \lambda_\ieps^4 = \frac{\eps}{4 \pi \sqrt 3}  \left( \int_\Omega V \mathcal G^2 \diff x  + o(1) \right). 
\end{equation} 
If $\int_\Omega V \mathcal G^2 \diff x \neq 0$, by passing to the limit $\eps \to 0$ in \eqref{final identity} and recalling that $\bm{\lambda}_0 = \bm{\Lambda}_0$ by Lemma \ref{lemma perron frobenius}, we clearly obtain \eqref{mu asymptotics thm}. If $\int_\Omega V \mathcal G^2 \diff x = 0$, suppose that $\sum_i a(x_{i,0})  \lambda_{i,0}^4 < 0$, say. Then the right side of \eqref{final identity} must be strictly negative. Hence the quotient $(\sum_i (a(x_\ieps) +o(1)) \lambda_\ieps^4) ( \int_\Omega V \mathcal G^2 \diff x  + o(1))^{-1}$ is positive and tends to $+ \infty$ as $\eps \to 0$. When $\sum_i a(x_{i,0})  \lambda_{i,0}^4 > 0$, the argument is analogous.

The proof of Theorem \ref{theorem multibubble} is thus complete.

\subsection{Regularity of $\phi_a$ and $\rho_a$}

We end this section by discussing sufficient conditions for $C^2$-differentiability and real-analyticity of the Robin function $\phi_a$, and, in turn, $\rho_a$. (We will use the terms \emph{analytic} and \emph{real-analytic} interchangeably in the following.)

We first observe some sufficient conditions for regularity of $\rho_a$.  

\begin{lemma}
\label{lemma rho_a analytic}
\begin{enumerate}[(i)]
\item If $\phi_a$ and $a$ are real-analytic on $\Omega$ then $\rho_a$ is real-analytic on $\Omega^n_*$. 
\item If $a \in C^{0,1}(\overline{\Omega}) \cap C_\text{loc}^{2,\sigma}(\Omega)$ for some $\sigma \in (0,1)$, then $\phi_a \in C^2(\Omega)$ and $\rho_a \in C^2(\Omega^n_*)$. 
\end{enumerate}

\end{lemma}

In the statement of Lemma \ref{lemma rho_a analytic}, we chose to assume global analyticity of $a$ and $\phi_a$ for simplicity. Since analyticity is a local property, it would of course be equally possible to conclude analyticity of $\rho_a$ in a neighborhood of some $\bm{x}_0 = (x_1,...,x_n)$ by assuming $a$ and $\phi_a$ to be analytic on neighborhoods of $x_1,...,x_n$.

\begin{proof}
It is a general fact that if a matrix $M(\xi)$ depends analytically on some parameter $\xi \in \R^n$ in a neighborhood of $\xi_0 \in \R^n$, then its simple eigenvalues also depend analytically on these parameters. This is a direct consequence of the analytic implicit function theorem applied to $p(\xi, \lambda) := \det(M(\xi) - \lambda \,  \text{Id})$ and the fact that $p(\xi_0, \lambda_0) = 0 \neq \partial_\lambda p(\xi_0, \lambda_0))$ if and only if $\lambda_0$ is a simple eigenvalue of $M(\xi_0)$.

We apply this fact to the matrix $M(\bm{x}$ with parameter $\bm{x} \in \Omega^n_*$. The off-diagonal entries $G_a(x_i,x_j)$ are always analytic in $x_i$, $x_j$ because $x_i \neq x_j$. This follows from elliptic regularity and the fact that $G_a(x,y)$ solves the PDE $-\Delta_y G_a(x,y) =-  a(x) G_a(x,y)$ on  $\Omega \setminus \{x\}$ with analytic coefficient $a(x)$. Now by assumption, the diagonal entries of $M(\bm{x})$ also depend analytically on $\bm{x}$. This completes the proof of (i). 

For (ii), we know from \cite[Lemma 4.1]{Frank2021b} that $\phi_a \in C^2(\Omega)$. Then the $C^2$-differentiability follows as above using the implicit function theorem formulated for $C^2$ functions. 
\end{proof}

In the simplest case where $a$ is a constant, we can prove that the hypothesis of Lemma \ref{lemma rho_a analytic}(i) is indeed fulfilled. 

\begin{lemma}
\label{lemma phi a is analytic}
Suppose that $a \equiv \text{const.}$. Then $\phi_a$ is real-analytic on $\Omega$ and $\rho_a$ is real-analytic on $\Omega^n_*$. 
\end{lemma}

\begin{proof}
We write $H_a(x,y)$ as  
\begin{equation}
\label{infinite sum H a}
H_a(x,y) = \eta (x,y) +  \sum_{k = 0}^\infty h_k(x,y), 
\end{equation} 
for some sequence of functions $h_k(x,y)$ satisfying 
\begin{equation}
\label{h k recursion start}
-\Delta_x h_0(x,y) = \frac{a}{4 \pi |x-y|}  \qquad \text{ for } x, y \in \Omega, 
\end{equation} 
and recursively, for $k \geq 1$, 
\begin{equation}
\label{h k recursion}
-\Delta_x h_k(x,y) = - a h_{k-1}(x,y) \qquad \text{ for } x, y \in \Omega. 
\end{equation} 
It is easy to verify that $h_k(x,y)$ given by 
\begin{equation}
\label{h k}
h_k(x,y) = \frac{a^{k+1} }{4 \pi (2k+2)!} |x-y|^{2k+1}, 
\end{equation} 
satisfies \eqref{h k recursion start}--\eqref{h k recursion}. In particular, with this choice the sum in \eqref{infinite sum H a} converges indeed on all of $\R^3$. By construction, the remainder function $\eta(x,y)$ satisfies 
\[ -\Delta_x  \eta(x,y) + a(x) \eta(x,y) = 0. \]
Since $a$ is analytic, elliptic regularity theory implies that $\eta(x,y)$ is analytic as a function of $x$. On the other hand, $\eta(x,y)$ is symmetric in $x$ and $y$ because $H_a(x,y)$ is and so are all the $h_k(x,y)$ by their explicit expressions given in \eqref{h k}. Hence as a function of $y$, $\eta(x,y)$ satisfies
\[ -\Delta_y \eta(x,y) +a(y) \eta(x,y) = 0, \]
and as above we conclude that $\eta(x,y)$ is analytic in $y$. 

Since 
\begin{equation}
\label{h k zero on diagonal}
h_k(x,x) = 0
\end{equation}
 by \eqref{h k}, from \eqref{infinite sum H a} we obtain $\phi_a(x) = \eta(x,x)$
and hence $\phi_a$ is analytic. 
\end{proof}

The above proof, in particular the ansatz given by \eqref{infinite sum H a}--\eqref{h k recursion}, still appears to be a promising approach to prove  analyticity of $\phi_a$ in the more general case where $a$ is analytic, but non-constant. However, there are several obstructions to a straightforward adaptation. Firstly,  the $h_k(x,y)$ will not have a simple expression as in \eqref{h k} because of additional terms coming from derivatives of $a$. In particular, it is much harder to find a way to simultaneously justify \eqref{h k zero on diagonal} and the convergence of the sum in \eqref{infinite sum H a}. Secondly, it is not clear how to prove analyticity of $\eta(x,y)$ in the second variable $y$. In particular, the symmetry of $h_k$ in $x$ and $y$ seems problematic to ensure. 

\appendix

\section{ A Liouville type result for the solutions of the linearized equation}
\label{A}
In this section, we give a general statement for the classification of the solution of the linearized equation \eqref{lin eq}. This kind of result is not new, see for instance \cite[Proposition 2]{KMPS} or \cite[Proposition 3]{MP}. The cited results concern solutions to \eqref{lin eq} which are singular in the origin, and the linearization takes place, somewhat more generally, about a singular solution to $-\Delta u = u^\frac{N+2}{N-2}$ instead of the bubble. Nevertheless the proof of  those former results can also be applied in our framework. For the sake of completeness and accessibility, we now make a precise statement in our context and give a quick sketch of its proof.

We consider the equation
\begin{equation}
\label{lin eq}
-\Delta v = N(N+2)  \tilde{B}^{p-1} v \qquad \text{ on } \R^N,
\end{equation}
with $p = \frac{N+2}{N-2}$ and $\tilde{B}(x) = \left( 1 + |x|^2\right)^{-\frac{N-2}{2}}$. Plainly, \eqref{lin eq} is the linearization of $-\Delta u = N(N-2) u^p$ at the solution $\tilde{B}$. (Note that the normalization of the bubble we choose here is different from the one employed in Theorem \ref{theorem multibubble}. This turns out to be  more natural and convenient for some of the related functions appearing below.) 

The canonical solutions to \eqref{lin eq} are linear combinations of the functions
\begin{equation}
\label{psi 0 definition}
w_0(x):= \frac{1-|x|^2}{(1+|x|^2)^{N/2}} = \frac{2}{2-N} \partial_\mu \left(\mu^{-\frac{N-2}{2}} \tilde{B}(\mu^{-1} x)\right) \Big|_{\mu=1}
\end{equation}
and 
\begin{equation}
\label{psi i definition}
w_i(x) := \frac{x_i}{(1+|x|^2)^{N/2}} = \frac{1}{N-2} \partial_{y_i} \left( \tilde{B}(x - y) \right)\Big|_{y=0},  \qquad i = 1,...,N,
\end{equation}
which arise as derivatives of $\tilde{B}$ with respect to its symmetry parameters. 

The main non-degeneracy result reads as follows. 

\begin{proposition}
\label{theorem non degeneracy intro}
Let $v$ be a solution to \eqref{lin eq} and suppose that $|v(x)| \lesssim |x|^\tau$ for all $|x| \geq 1$, for some $\tau \geq -N+2$. 
Then 
\begin{equation}
\label{u linear combi theorem}
v(x) =  \sum_{i=0}^3 c_i w_i(x) + \sum_{k= 2}^{\lfloor \tau \rfloor} v_k^-(r) Y_k(\omega)
\end{equation}
for some smooth functions $v_k^-$ defined on $(0, \infty)$. Here $r = |x|$, $\omega = x/|x|$ and $Y_k$ is a spherical harmonic on $\mathbb S^{N-1}$ of degree $k$. 

Moreover, for every $k \geq 2$, we have  either $v_k^-\equiv 0$  or $v_k^-(r) \sim r^k$ both as $r \to 0$ and as $r \to \infty$. 

%Conversely, any function of the form \eqref{u linear combi theorem} solves \eqref{lin eq}. 
\end{proposition}

\begin{proof}
Since $-\Delta$ is diagonal with respect to spherical harmonics and $\tilde{B}$ is radial, we can write a solution to \eqref{lin eq} as 
\[ v = \sum_{k = 0}^\infty v_k(r) Y_k(\theta), \]
where $Y_k$ is a suitable spherical harmonic of degree $k$ and $v_k$ solves the equation 
\begin{equation}
\label{v_k equation}
v''_k + \frac{N-1}{r} v'_k + \left(N(N+2) \tilde{B}^{p-1} - \frac{k (k + N-2)}{r^2} \right) v_k = 0. 
\end{equation} 
It follows from the discussion on \cite[p. 310-311]{MPU}, that the only solutions for degrees $k = 0,1$ which satisfy $|v_k(r)| \lesssim r^\tau$ are constant multiples of $v_0(r) = \frac{1 - r^2}{(1 + r^2)^{N/2}}$ and $v_1(r) = \frac{r}{(1 + r^2)^{N/2}}$.

For $k \geq 2$, passing to logarithmic coordinates via 
\begin{equation}
\label{log change of var}
v_k(r) = r^{-\frac{N-2}{2}} \psi_k (\ln r), 
\end{equation} 
the new unknown function $\psi_k$ satisfies 
\begin{equation}
\label{lin log}
\mathcal L_k \psi_k := \psi''_k(t) - \mu_k^2 \psi_k(t) + g(t) \psi_k(t)  = 0, 
\end{equation} 
where we have set 
\[ 
\mu_k := \frac{N-2}{2} + k \quad \text{ and } \quad g(t) := \frac{N(N+2)}{4} \cosh(t)^{-2}.
\] 
It follows from arguments developed in \cite{MPU}, see also \cite[Section 2.2]{KMPS}, that for each $k \geq 2$ there are precisely two linearly independent solutions $\psi_k^\pm$ to \eqref{lin log}, satisfying $|\psi_k^\pm(t)| \sim e^{\mp \mu_k t}$ for $t \in \R$. (This could alternatively also be proved via elementary ODE analysis.) 

The solutions $v_k^\pm$ associated to $\psi_k^\pm$ via \eqref{log change of var} then satisfy, if not identically equal to $0$,
\begin{equation}
\label{bounds v_k} 
|v_k^-(r)|  \sim r^k, \qquad |v_k^+(r)| \sim r^{-N + 2 - k} \qquad \text{ for } r > 0. 
\end{equation}
In particular, $v_k^+$ is singular near zero. Since $v$ is bounded near the origin by assumptions, $v_k$ must be proportional to $v_k^-$ for every $k \geq 2$. Moreover, because of the growth bound $|v(x)| \lesssim |x|^\tau$ as $|x| \to \infty$, \eqref{bounds v_k} forces $v_k^- \equiv 0$ for every $k > \lfloor \tau \rfloor$. This yields the claimed expression for $v$. 
\end{proof}

%Note that we can express 
%\begin{equation}
%\label{psi 0 psi 1 in alpha Y}
%v_0(x) = \bar{\alpha}_0(r) Y_0(\omega) \quad \text{ and } \quad w_i(x) = v_1(r) Y_1^{(i)}(\omega). 
%\end{equation}
%Here $v_0(r)= \frac{1-r^2}{(1+r^2)^{N/2}}$ and $v_1(r) = \frac{r}{(1+r^2)^{N/2}}$ are the solutions to \eqref{ODE alpha k} for $k = 0,1$, and $Y_0(\omega) = 1$, $Y_1^{(i)}(\omega) = \omega_i$.  On the other hand, when $k \geq 2$, we are not aware of a simple explicit expression of the $\bar{\alpha}_k$. 

In our proof of Theorem \ref{theorem multibubble}, we  use the non-degeneracy statement of Proposition \ref{theorem non degeneracy intro} in the form of the following corollary.

\begin{corollary}
\label{corollary u = 0}
Let $v$ be a solution to \eqref{lin eq} and suppose that $|v(x)| \lesssim |x|^\tau$ on $\R^{N}$ for some $\tau \in (1, \infty) \setminus \N$. Then $v \equiv 0$. 
\end{corollary}

\begin{proof}
By Proposition \ref{theorem non degeneracy intro}, $v$ is of the form \eqref{u linear combi theorem}. But now it is easy to see that the assumption $v(x) \lesssim |x|^\tau$ as $|x| \to 0$ with $\tau > 1$, together with \eqref{bounds v_k}, forces $c_i = 0$ for $i = 0,...,N$. If $\tau < 2$, we are done. If $\tau > 2$, we have 
\begin{equation}
\label{bound tau contradiction}
\infty > C \geq \lim_{|x| \to 0} \frac{|v(x)|}{|x|^\tau} = \sum_{k = 2}^\infty  \frac{|v_k^-(r)|}{r^\tau} Y_k(\theta). 
\end{equation}
Now, if some $v_k^-$ is not identically equal to $0$, then we consider $k_0$ smallest $k$ such that $v_k^-\not \equiv 0$. Since $\tau$ is non-integer, $k_0<\tau$, and we get, by \eqref{bounds v_k}, 
$$
\frac{|v(x)|}{|x|^\tau} \sim \vert x\vert^{k_0 -\tau} \text{ at } 0.
$$
This yields a contradiction with \eqref{bound tau contradiction}, hence $v \equiv 0$ as claimed.
\end{proof}

The preceding proof also shows that the restriction $\tau \notin \N$ is necessary for Corollary \ref{corollary u = 0} to hold. Indeed, for any $k \in \N$ and any spherical harmonic $Y_k$ of degree $k$, the function $v(x) = v^-_k(r) Y_k(\omega)$ satisfies  \eqref{lin eq} with $|v(x)| \lesssim |x|^k$ for all $x \in \R^N$, while certainly $v \not \equiv 0$.

\section{Some computations}

\begin{lemma}
\label{lemma int B^5 Ga}
Let $a \in C(\overline{\Omega}) \cap C^{1, \sigma}_\text{loc}(\Omega)$ for some $\sigma > 0$.  As $\eps \to 0$, 
\[ \int_{\mathsf{b}_\ieps} B_\ieps^5 G_a(x_\ieps, \cdot) \diff x = \mu_\ieps^{-1/2} - 4 \pi \sqrt 3 \phi_a(x_\ieps) \mu_\ieps^{1/2} + 3 a(x_\ieps) \mu_\ieps^{3/2} + o(\mu_\eps^{3/2}). \]
\end{lemma}

\begin{proof}
Under our assumptions on $a$, \cite[Lemma B.2]{Frank2021b} asserts that
\begin{equation}
\label{Ha expansion}
G_a(y,z) = \frac{1}{4\pi |z-y|} -  \phi_a(y)  - \frac 12 \nabla \phi_a(y)\cdot (z-y) + \frac{a(y)}{8 \pi}|z-y| + \mathcal O(|z-y|^{1+\nu}),
\end{equation} 
for every $0 <\nu < 1$.
Recall ${\mathsf{b}_\ieps} = B(x_\ieps, \delta_0)$ with $\delta_0 > 0$ independent of $\eps$, and pick $y = x_\ieps$ in \eqref{Ha expansion}. We compute
\begin{align*}
\int_{\mathsf{b}_\ieps} \frac{1}{4\pi |z-x_\ieps|} B_\ieps^5 \diff z = \frac{1}{4 \pi} \mu_\ieps^{-1/2} \int_{B(0, \mu_\ieps^{-1} \delta_0)} B^5 \frac{1}{|z|} \diff z = \mu_\ieps^{-1/2} + o(\mu_\ieps^{3/2}).
\end{align*}
Here we used that 
\[ \int_{\R^3} \frac{1}{|x|} \frac{1}{(1+ \frac{|x|^2}{3})^{5/2}} \diff x = 6 \pi \int_0^\infty \frac{1}{(1 + s)^{5/2}} \diff s = 6 \pi B(1, 3/2) = 4 \pi, \]
where $B(a,b) = \frac{\Gamma(a)\Gamma(b)}{\Gamma(a+b)}$ is the Beta function. 

Next, 
\begin{align*}
 \phi_a(x_\ieps) \int_{\mathsf{b}_\ieps} B_\ieps^5 \diff z &= \mu_\ieps^{1/2} \phi_a(x_\ieps) \mu_\ieps^{1/2}  \int_{B(0, \mu_\ieps^{-1} \delta_0)} B^5 \diff z \\
 &= 4 \pi \sqrt 3 \phi_a(x_\ieps) \mu_\ieps^{1/2} + o(\mu_\eps^{3/2}). 
\end{align*}
Here we used that 
\[ \int_{\R^3} \frac{1}{(1+ |x|^2/3)^{5/2}} \diff x = 6 \pi \sqrt 3 \int_0^\infty \frac{s^{1/2}}{(1 + s)^{5/2}} \diff s = 6 \pi \sqrt 3 B(3/2, 1) = 4 \pi \sqrt 3. \]

By antisymmetry of the integrand,
 \[ \int_{\mathsf{b}_\ieps} \nabla \phi_a(x_\ieps) \cdot (z-x_\ieps) B_\ieps^5 \diff z = 0. \]
Next, 
\begin{align*}
\frac{a(x_\ieps)}{8 \pi} \int_{\mathsf{b}_\ieps} |z - x_\ieps| B_\ieps^5 \diff x &= \frac{a(x_\ieps)}{8 \pi} \mu_\ieps^{3/2} \int_{B_{\mu_\ieps^{-1} \delta}(0)} B^5 |x| \diff x \\
&= 3 a(x_\ieps) \mu_\ieps^{3/2} + o(\mu_\ieps^{3/2}).
\end{align*}
Here we used that 
\[ \int_{\R^3} |x| \frac{1}{(1+ |x|^2/3)^{5/2}} \diff x = 18 \pi \int_0^\infty \frac{s}{(1 + s)^{5/2}} \diff s = 18 \pi B(2, 1/2) = 24 \pi. \]
Finally, 
\[ \int_\Omega B_\ieps^5 |z-x_\ieps|^{1 + \nu} \diff x \lesssim \mu_\eps^{\frac{3}{2} + \nu} = o(\mu_\eps^{3/2}). \]
Combining all of the above, the lemma follows. 
\end{proof}

\begin{lemma}
\label{lemma int B^5 nabla Ga}
Let $a \in C(\overline{\Omega}) \cap C^{1, \sigma}_\text{loc}(\Omega)$ for some $\sigma > 0$. As $\eps \to 0$, we have 
\[ \int_{\mathsf{b}_\ieps} B_\ieps^5 \nabla_x G_a(x_\ieps,z) \diff z =  - 2 \pi \sqrt 3 \nabla \phi_a(x_\ieps) \mu_\ieps^{1/2} + \mathcal O(\mu_\eps^{\frac{1}{2} + \nu}), \]
for every $0 < \nu < 1$. 
\end{lemma}

\begin{proof}
The argument in \cite[Lemma B.2]{Frank2021b} in fact also shows
\[ 
\nabla_x G_a(y,z) = \frac{y-z}{4\pi |y-z|^3} - \frac 12 \nabla \phi_a(z)  + \frac{a(z)}{8 \pi} \frac{y-z}{|y-z|} + \mathcal O(|y-z|^{\nu}),
\]
for every $0 < \nu < 1$. Picking $y = x_\ieps$, and observing the cancellations by antisymmetry, this identity gives 
\begin{align*}
&\qquad \int_{\mathsf{b}_\ieps} B_\ieps^5(z) \nabla_x G_a(x_\ieps, z) \diff z \\
&= - \frac 12  \int_{\mathsf{b}_\ieps} B_\ieps^5(z) \nabla \phi_a(z) \diff z + \mathcal O \left( \int_{\mathsf{b}_\ieps}  B_\ieps^5(z) |z - x_\ieps|^\nu \diff z \right) \\
&= - 2 \pi \sqrt 3 \nabla \phi_a(x_\ieps)+ \mathcal O(\mu_\eps^{\frac 12  + \nu}). 
\end{align*}
This is the assertion. 
\end{proof}

\begin{lemma}
\label{lemma nabla phi a + eps V}
Let $\eps > 0$, $a \in C(\overline{\Omega})$ and $V \in C^(\overline{\Omega}) \cap C^{0,\sigma}_\text{loc}(\Omega)$ for some $sigma \in (0,1)$ be such that the Green's functions $G_a$ and $G_{a + \eps V}$ exist. Then 
\begin{equation}
\label{phi a resolvent}
\phi_{a + \eps V}(x) - \phi_a(x) = \eps \int_\Omega G_a(x,y)^2 V(y) \diff y + \mathcal O(\eps)
\end{equation} 
and 
\begin{align}
\label{nabla phi a resolvent}
\nabla \phi_{a + \eps V}(x) - \nabla \phi_a(x) &= \mathcal O(\eps).
\end{align}
The bounds are uniform for $x$ in compact subsets of $\Omega$. 
\end{lemma}

\begin{proof}
By the resolvent formula, we have
\begin{equation}
\label{resolvent general}
H_{a + \eps V}(x,y) -H_a(x,y) = G_a(x,y) - G_{a + \eps V}(x,y) = \eps \int_\Omega G_a(x,z) V(z) G_{a + \eps V}(z,y) \diff z. 
\end{equation} 
In particular, $G_{a + \eps V}(x,y) = G_a(x,y) + \mathcal O(\eps)$. Plugging this back into the right side of \eqref{resolvent general} and evaluating at $x = y$ gives \eqref{phi a resolvent}. 

To prove \eqref{nabla phi a resolvent}, some more care needs to be taken because the derivative of the integrand in \eqref{phi a resolvent} behaves like $|x-y|^{-3}$, which is not integrable. To overcome this issue, we use the regularity of $V$. Indeed, by decomposing $G_a(x,y) = \frac{1}{4\pi |x-y|} - H_a(x,y)$, it is easy to see that 
\[ \nabla_x \int_\Omega G_a(x,y) G_{a + \eps V} V(y) \diff y = \nabla_x \int_{B_d} |x-y|^{-2} V(y) \diff y + \mathcal O(1), \]
where $B_d$ is a ball of radius $d = \text{dist}(x, \partial \Omega)$ around $x$. But 
\begin{align*}
&\qquad \left| \nabla_x \int_{B_d} |x-y|^{-2} V(y) \diff y \right| \\
 & = 2 \left| \int_{B_d} \frac{x-y}{|x-y|^4} V(y) \diff y \right| =   2 \left| \int_{B_d} \frac{x-y}{|x-y|^4} (V(x) - V(y)) \diff y \right| \\
&\lesssim \int_{B_d} |x-y|^{-3 + \sigma} \diff y < \infty. 
\end{align*} 
The differentation under the integral in this computation is slightly formal because $ \frac{x-y}{|x-y|^4}$ is not absolutely integrable, but can easily be made rigorous using a cutoff around the singularity. We omit the details. 
\end{proof}

\begin{lemma}
\label{lemma integral w0}
Let $W$ be the unique radial solution to 
\[ -\Delta W - 5 W B^4 = -B, \qquad W(0) = \nabla W(0) = 0. \]
Then we have 
\[ 5 \int_{\R^3} \frac{W(x) B(x)^4}{|x|} \diff x = 12 \pi   (\pi - 1). \]
\end{lemma}

\begin{proof}
By the equation, we have, for every $R > 0$,  
\begin{equation}
\label{w0 eq integrated}
5 \int_{B_R} \frac{W(x) B(x)^4}{|x|} \diff x = - \int_{B_R}  \frac{\Delta W}{|x|} \diff x +  \int_{B_R} \frac{B(x)}{|x|} \diff x. 
\end{equation} 
We need to compute the asymptotics as $R \to \infty$ of the two integrals on the right side. The second one is straightforward to evaluate. We have 
\begin{align}
\int_{B_R} \frac{B(x)}{|x|} \diff x &= 4 \pi \int_0^R \frac{r \diff r}{(1 + \frac{r^2}{3})^{1/2}} = 4 \pi \left[3(1 + \frac{r^2}{3})^{1/2}\right]^R_0  \nonumber \\
& = 12 \pi (1 + \frac{R^2}{3})^{1/2} - 12 \pi = 4 \pi \sqrt 3 R - 12 \pi + \mathcal O(R^{-1}) \label{B/x expansion}
\end{align}
as $R \to \infty$. To evaluate the first integral on the right side of \eqref{w0 eq integrated}, we integrate by parts. By Green's formula and since $W(0) = 0$, we have 
\begin{align}
 - \int_{B_R}  \frac{\Delta W}{|x|} \diff x &= \int_{\partial B_R} \left( W \frac{\partial}{\partial n} \frac{1}{|x|} - \frac{1}{|x|} \frac{\partial W}{\partial n} \right) \diff \sigma(x) \nonumber \\
 &= - 4\pi W(R) - 4\pi W'(R) R.  \label{w0 at R}
\end{align}
By Lemma \ref{lemma w0 variation of constants} below, we have 
\[ W(r) = v(r) \varphi(r), \]
where $v$ is the solution to the homogeneous equation $-\Delta v = 5 v B^4$ given by 
\begin{equation}
\label{v0 definition}
v(r) = \frac{3 - r^2}{(3 + r^2)^{3/2}} = - \frac{2}{\sqrt 3} \frac{\diff}{\diff \mu}|_{\mu = 1} B_{\mu,0}(r) 
\end{equation} 
and $\varphi(r) = \int_0^r \psi(s) \diff s$, with 
\begin{equation}
\label{psi definition}
\psi(r) =\sqrt 3 \left(-r + 2 \sqrt 3 \arctan \left( \frac{r}{\sqrt 3} \right) - \frac{3r}{r^2 + 3} \right) \frac{(3 + r^2)^3}{r^2(3-r^2)^2}. 
\end{equation} 
From these expressions, we can easily read off the asymptotic behavior of $W(R)$ and $W'(R)$ to the precision necessary to evaluate \eqref{w0 at R} as $R \to \infty$. Indeed, we have 
\[ \psi(R) = -  \sqrt 3 R + 3 \pi  + \mathcal O(R^{-1}), \]
hence 
\[ \varphi(R) = - \frac{ \sqrt 3}{2} R^2 + 3 \pi  R + \mathcal O(\ln R). \]
On the other hand, 
\[ v(R) = - R^{-1} + \mathcal O(R^{-3}), \]
which yields 
\[ W(R) = \frac{ \sqrt 3}{2} R - 3 \pi  + o(1). \]
Moreover, 
\[ v'(R) = 3R \frac{R^2 - 3}{(3 + R^2)^{5/2}} = R^{-2} + \mathcal O(R^{-4}) \]
and thus 
\begin{align*}
W'(R) &= \psi(R) v(R) + \varphi(R) v'(R)  \\
& = ( \sqrt 3 - 3  \pi R^{-1} ) + (\frac{- \sqrt 3 }{2} + 3 \pi  R^{-1} ) + o(R^{-1}) \\
& = \frac{ \sqrt 3 }{2} + o(R^{-1}). 
\end{align*} 
Inserting these expansions into \eqref{w0 at R} above, we obtain
\[  - \int_{B_R}  \frac{\Delta W}{|x|} \diff x = - 4 \pi \sqrt 3  R + 12 \pi^2  + o(1). \]
Coming back to \eqref{w0 eq integrated} and inserting this expansion as well as \eqref{B/x expansion}, the divergent terms in $R$ cancel and we get 
\[ 5 \int_{\R^3} \frac{W(x) B(x)^4}{|x|} \diff x = 5 \lim_{\R \to \infty} \int_{B_R} \frac{W(x) B(x)^4}{|x|} \diff x = 12 \pi   (\pi - 1), \]
as claimed.
\end{proof}

\begin{lemma}
\label{lemma w0 variation of constants}
Let $W$ be the unique radial solution to 
\begin{equation}
\label{w0 equation lemma}
 -\Delta W - 5 W B^4 = -B, \qquad W(0) = \nabla W(0) = 0. 
\end{equation}
Then $W$ is given by 
\begin{equation}
\label{w0 definition} W(r) = v(r) \int_0^r \psi(s) \diff s, 
\end{equation}
with $v$ as in \eqref{v0 definition} and $\psi$ as in \eqref{psi definition}. 
\end{lemma}

Notice that indeed $W'(0) = 3^{-1/2} \psi(0) = 0$. This is not directly obvious from the definition of $\psi$, but follows by noting that for 
\[ h(r):=  -r + 2 \sqrt 3 \arctan \left( \frac{r}{\sqrt 3} \right) - \frac{3r}{r^2 + 3} \]
one has $h(0) = h'(0) = h''(0) = 0$. This implies $h(r) = \mathcal O(r^3)$ and thus $\psi(r) = \mathcal O(r)$ as $r \to 0$.

It can of course be verified by straightforward computation that $W$ given by \eqref{w0 definition} solves \eqref{w0 equation lemma}. In the following proof we actually sketch how to find \eqref{w0 definition} using the method of the variation of constants.

\begin{proof}
Setting $W = v \varphi$ with the new unknown $\varphi$, solving \eqref{w0 equation lemma} becomes equivalent to solving 
\[ \varphi'' + 2 (\frac{1}{r} + \frac{v'}{v}) \varphi' = \frac{ B}{v}, \qquad \varphi(0) = \varphi'(0) = 0, \]
or, with $\varphi(r) = \int_0^r \psi(s) \diff s$ and $H:= 2 (\frac{1}{r} + \frac{v'}{v})$, 
\[ \psi' + H \psi = \frac{ B}{v}, \qquad \psi(0) =0, \]
To solve this first-order equation, we make a second time the variation of constants ansatz $\psi = \psi_0 \eta$, where $\psi'_0 + H \psi_0 = 0$ and $\eta$ needs to solve 
\[ \eta' =  \frac{B}{v \psi_0}, \quad \eta(0) = 0. \]
The solution $\psi_0$ can be determined directly as 
\[ \psi_0(r) = \exp \left(-\int_1^r H(s) \diff s\right) = \frac{1}{r^2 v^2}, \]
and hence
\[ \eta(r) =  \int_0^r B(s) s^2 v(s) \diff s = \sqrt 3 \int_0^r \frac{s^2 (3 - s^2)}{(3 + s^2)^2} \diff s = 3 \int_0^{r/\sqrt 3 } \frac{s^2 (1 - s^2)}{(1+ s^2)^2} \diff s. 
 \]
To evaluate this integral, we write 
\[ \frac{s^2 (1 - s^2)}{(1+ s^2)^2} = - 1 + \frac{3s^2 +1}{(1+s^2)^2} = - 1 + \frac{3}{1 + s^2} - \frac{2}{(1 + s^2)^2}. \]
It can be verified by direct computation that $(\arctan s + \frac{s}{s^2 +1} )' = \frac{2}{(1 + s^2)^2}$. From here, we can thus explicitly compute $\eta$, and via $\psi = \psi_0 \eta$ we easily obtain the claimed expression for $\psi$. 
\end{proof}

\begin{lemma}
\label{lemma Wjk}
For $j,k \in \{1,2,3\}$, there are functions $W_{jk}$ which satisfy
\begin{equation}
\label{w equation lemma}
\begin{cases}
-\Delta W_{jk} - 5 W_{jk} B^4 = 0 , \quad W_{jk}(x) = x_j x_k + o(|x|^2) \quad \text{ as } x \to 0 & \text{ if } j \neq k, \\
-\Delta W_{jj} - 5 W_{jj} B^4 = -B , \quad W_{jj}(x) = \frac{1}{2} x_j^2 + o(|x|^2) \quad \text{ as } x \to 0 & \text{ if } j = k.
\end{cases}
\end{equation} 
\end{lemma}

\begin{proof}
If $j \neq k$, we make the ansatz $W_{jk}(x) = f(|x|) Y_{jk}(x/|x|)$, with $Y_{jk}(\omega) = \omega_j \omega_k$ for $\omega \in \mathbb S^2$. Observing that $Y_{jk}$ is a spherical harmonic of degree $2$, $W_{jk}$ solves the equation in \eqref{w equation lemma} if and only if $f$ solves the ODE 
\begin{equation}
\label{ode proof}
- f''(r) - \frac{2}{r} f'(r) + \frac{6}{r^2} f(r)  + 5 f(r) B^4(r) = 0 \quad \text{ on } (0, \infty). 
\end{equation} 
Following the discussion in the proof of Proposition \ref{theorem non degeneracy intro},  there is a solution $f$ ($f = v_2$ in the notation of Proposition \ref{theorem non degeneracy intro}, up to the slightly different normalization of $B$) which satisfies $f(r) \sim r^2$ for $r \in (0, \infty)$. Up to replacing $f$ by a suitable scalar multiple, we may thus assume that $\lim_{r \to 0} f(r) r^{-2} = 1$. It follows that 
\[ W_{jk}(x) = f(|x|) Y_{jk}(x /|x|) = f(|x|) |x|^2 x_j x_k = (1 + o(1)) x_j x_k = x_j x_k + o(|x|^2). \]

If $j = k$, we set
\[ W_{jj}(x) = f(|x|)Y_j(x/|x|) + W(|x|), \]
where $Y_j(\omega) = \frac{1}{2} \omega_j^2 - \frac{1}{6}$, $f$ is a solution to \eqref{ode proof} with $\lim_{r \to 0} f(r) r^{-2} = 1$ and $W$ is the function from Lemma \ref{lemma integral w0}. Observing that $Y_j$ is a spherical harmonic of degree $2$, $W_{jj}$ satisfies the equation in \eqref{w equation}. Moreover, 
\[ W_{jj}(x) = f(|x|) \left(\frac{1}{2} \frac{x_j^2}{|x|^2} - \frac{1}{6} \right) + W(|x|) = \frac{1}{2} x_j^2 + \left( W(|x|) - \frac{1}{6} f(|x|) \right) + o(|x|^2) . \]
From the expression for $W$ found in Lemma \ref{lemma w0 variation of constants}, it is tedious but straightforward to check that $W(r) = \frac{1}{6} r^2$ as $r \to 0$. Hence $ W(|x|) - \frac{1}{6} f(|x|) = o(|x|^2)$, and the proof is complete. 
\end{proof}

\bibliography{brezis-peletier}
	\bibliographystyle{plain}

\end{document}